\documentclass{amsart}

\usepackage[utf8]{inputenc}

\usepackage{geometry}
\usepackage{graphicx}
\usepackage{xcolor}
\usepackage[hidelinks]{hyperref} 
\usepackage{cleveref}
\usepackage{amssymb}
\usepackage{bbm}

\usepackage[alphabetic, abbrev]{amsrefs}
\renewcommand\MR[1]{}

\newcommand\ti\tilde
\renewcommand\o\otimes
\newcommand\x\times

\renewcommand\:\colon
\newcommand{\ol}[1]{\overline{#1}}
\newcommand\cross{\mathsf{X}}

\newcommand\m\mu
\renewcommand\l\lambda
\renewcommand\k\kappa
\newcommand\s\sigma
\renewcommand\t\tau
\renewcommand\L\Lambda
\newcommand\g\gamma
\newcommand\G\Gamma
\newcommand\eps\varepsilon
\newcommand\D\Delta

\newcommand\U{\mathcal{U}}
\newcommand\kk{\mathbbm{k}}
\newcommand\gf{\mathfrak{g}}
\newcommand\gl{\mathfrak{gl}}
\newcommand\cC{\mathcal{C}}

\newcommand\J{\mathcal{J}}
\newcommand\one{\mathbbm{1}}
\newcommand\Ug{{\U\gf}}
\newcommand\NN{\mathbb{N}_0}
\newcommand\CC{\mathbb{C}}

\newcommand\gr{\operatorname{gr}}
\newcommand\id{\operatorname{id}}
\newcommand\im{\operatorname{im}}

\newcommand\alt{{\operatorname{alt}}}
\newcommand\sym{{\operatorname{sym}}}
\newcommand\cyc{{\operatorname{cyc}}}
\newcommand\Hom{\operatorname{Hom}}

\newcommand\Ind{\operatorname{Ind}}
\newcommand\Aut{\operatorname{Aut}}
\newcommand\Rep{\operatorname{Rep}}
\newcommand\ASD{\operatorname{ASD}}
\newcommand\OSp{\operatorname{OSp}}
\newcommand\sgn{\operatorname{sgn}}
\newcommand\Sp{\operatorname{Sp}}
\newcommand\SVec{\operatorname{SVec}}
\renewcommand\Vec{\operatorname{Vec}}

\newcommand\tr{\operatorname{tr}}
\newcommand\str{\operatorname{str}}

\newcommand\form{{\operatorname{form}}}
\newcommand\Lie{{\operatorname{Lie}}}

\newcommand\RepOt{\underline{\operatorname{Rep}}(O_t)}

\newcommand\tforall{\quad\text{ for all }}

\newcommand\teven{ \text{ even} }
\newcommand\todd{ \text{ odd} }
\newcommand\tand{ \text{ and } }

\newtheorem{theorem}{Theorem}[section]

\newtheorem{proposition}[theorem]{Proposition}
\newtheorem{corollary}[theorem]{Corollary}
\newtheorem{lemma}[theorem]{Lemma}

\newtheorem*{theorem*}{Theorem}

\theoremstyle{definition}
\newtheorem{definition}[theorem]{Definition}

\theoremstyle{remark}
\newtheorem{example}[theorem]{Example}
\newtheorem{remark}[theorem]{Remark}

\usepackage{tikz}
\usetikzlibrary{calc,positioning}

\newcommand\arch[1]{~\tikz[baseline=(n.base),rounded corners=8pt]{%
\node[inner xsep=5pt](n){$#1$};
\draw (n.base-|n.south west) |- (n.north) -| (n.base-|n.south east);
}~%
}

\newcommand\uarch[1]{~\tikz[baseline=(n.base),rounded corners=4pt]{%
\node[inner xsep=5pt](n){$#1$};
\draw ($(n.north west)+(0,-2pt)$) |- (n.south) -| ($(n.north east)+(0,-2pt)$);
}~%
}

\tikzset{
	partition/.style={
      scale=0.5,
      yscale=-1,
      baseline={([yshift=-0.5ex]current bounding box.center)}
    }
}

\tikzset{
    bend/.cd,
    0/.style={},
    1/.style={bend right},
    -1/.style={bend left}
}
\def\lenfac{1}
\newcommand\makePartPt[1]{({mod(#1,10)*\lenfac},{(#1-mod(#1,10))*.1})}
\newcommand\makePartLn[2]{%
\pgfmathtruncatemacro\bend{%
(int(#1/10)==int(#2/10)) ?
(#1<10 ? 1 : -1)*(#1>#2 ? 1 : -1)
: 0%
}
\draw ({mod(#1,10)*\lenfac},{int(#1/10)}) to[bend/\bend] ({mod(#2,10)*\lenfac},{int(#2/10)});
}

\newcommand\tpx[2] {%
\tikz[partition] {
\draw[white,opacity=0] (1,0)--(1,1); 

#2;

\def\j{0}
\foreach \i [remember=\i as \j] in {#1} {
  \ifdim \i cm > 0 cm
    \ifdim \j cm > 0 cm
      \makePartLn{\i}{\j};
    \fi
    \draw[fill] \makePartPt\i circle (2.5pt);
  \fi
} 
}}

\newcommand\tpa[4]{\tpx{#4}{
\foreach \i in {1,...,#1}
\draw[draw=black!20,line width=8pt,line cap=round] ({((\i-1)*(#2)+1)*\lenfac},0) -- ({\i*(#2)*\lenfac},0);
\foreach \i in {1,...,#3}
\draw[draw=black!20,line width=8pt,line cap=round] ({(2*\i-1)*\lenfac},1) -- ({2*\i*\lenfac},1);
}%
}

\newcommand\tp[1]{\tpx{#1}{}}

\newcommand\tpsyma[3]{\tpx{#3}{%
\draw[<->,gray] (#1*\lenfac,-0.5) -- (#1*\lenfac+\lenfac,-0.5);
\ifdim #2 cm > 0 cm
\draw[<->,gray] (\lenfac,1.5) -- (#2*\lenfac,1.5); 
\fi;
}}

\newcommand\tpsymb[2]{\tpx{#2}{%
\draw[<->,gray] (\lenfac,1.5) -- (2*\lenfac,1.5);
\ifdim #1 cm > 0 cm
\draw[<->,gray] (3*\lenfac,1.5) -- (#1*\lenfac+2*\lenfac,1.5); 
\fi;
}}

\tikzset{
	mydot/.style={
    draw=black!40, fill=black!20,
    circle, font={\sffamily\scriptsize}, inner sep=1.5pt
    }
}

\newcommand\mymoredots[2]{%
\foreach \i in {1,...,#1}
\node[mydot] (a\i) at (\i,1) {\i};
\ifnum #2 > 0
\foreach \i [count=\j from #1+1] in {1,...,#2}
\node[mydot] (b\i) at ({\i+(#1-1)*.5},0) {\j};
\fi
}

\newcommand\mydots[1]{\mymoredots{#1}{0}}

\newcommand\myloop[2]{%
\draw (#1) to[out=120,in=60,looseness=10] node[midway,above] {#2} (#1);
}

\newcommand\mydiagscale{0.8}
\newcommand\myxdots[2]{%
~\tikz[baseline=(a1.base),scale=\mydiagscale, every node/.style={transform shape}]{\mydots{#1}; \draw #2;}~%
}
\newcommand\myxmoredots[3]{%
~\tikz[baseline=0.5cm,scale=\mydiagscale, every node/.style={transform shape}]{\mymoredots{#1}{#2}; \draw #3;}~%
}

\numberwithin{equation}{section}

\title{Interpolating PBW Deformations for the Orthosymplectic Groups}

\author{Johannes Flake}
\address{Algebra and Representation Theory, RWTH Aachen University,
Pontdriesch 10--16,
52062 Aachen, Germany}
\email{flake@art.rwth-aachen.de}

\author{Verity Mackscheidt}
\address{Algebra and Representation Theory, RWTH Aachen University,
Pontdriesch 10--16,
52062 Aachen, Germany}
\email{mackscheidt@art.rwth-aachen.de}

\date{\today}

\begin{document}

\maketitle

\begin{abstract}
We propose to use interpolation categories to study PBW deformations, and demonstrate this idea for the orthosymplectic supergroups. Employing a combinatorial calculus based on pseudographs and partitions which we derive from a suitable Jacobi identity, we classify PBW deformations in (quotients of) Deligne's interpolation categories for the orthosymplectic groups. As special cases, our classification recovers families of infinitesimal Hecke algebras found by Etingof--Gan--Ginzburg (2005) for the symplectic groups and by Tsymbaliuk (2015) for the orthogonal groups together with their respective standard representations using completely different geometric methods. Our results can be viewed as an extension of these known results to the family of all orthosymplectic groups together with all of their fundamental representations, obtained by novel interpolation techniques for PBW deformations.
\end{abstract}

\section{Introduction}

PBW (Poincaré--Birkhoff--Witt) deformations are deformations of a filtered algebra preserving its associated graded algebra. The eponymous example is the universal enveloping algebra of a Lie algebra viewed as a deformation of the symmetric algebra, which enjoys this property according to the classical Poincaré--Birkhoff--Witt theorem. A general theory of PBW deformations was established by Bravermann--Gaitsgory \cite{BG},  Polishchuk--Positselski \cite{PP}, and in a more general form by Walton--Witherspoon \cite{WW}; a central result is the fact that PBW deformations are characterized by a suitable generalization of the classical Jacobi identity. Certain PBW deformations related to finite groups were investigated by Etingof--Ginzburg \cite{EG}, they include rational Cherednik algebras and symplectic reflection algebras. Both classes of algebras are instances of so-called Drinfeld--Hecke algebras which were studied in their own right (see \cite{RS}). Etingof--Gan--Ginzburg \cite{EGG} suggested to consider analogs of such Hecke algebras for algebraic groups, which they called continuous Hecke algebras and infinitesimal Hecke algebras. The classification of these PBW deformations in general is an open problem. However, certain classes of infinitesimal Hecke algebras for special choices of an algebraic group together with a representation were classified in \cite{EGG} and subsequent work. In particular, infinitesimal Hecke algebras for orthogonal and symplectic groups with their natural representation were classified in \cite{Tsy} and \cite{EGG}, respectively.

Deligne's interpolation categories \cite{Del} are (possibly nonabelian) tensor categories depending on a parameter, which interpolate the representation categories of families of finite or classical algebraic groups. Typically, they are universal categories with respect to an algebraic structure present in all representation categories within the family of groups. They can also be defined combinatorially via a graphical calculus based on string diagrams. Interpolation categories include interesting examples of tensor categories which do not admit fiber functors to categories of vector spaces, and are, hence, not representation categories of algebraic objects by a version of Tannakian reconstruction. Deligne's interpolation categories for the orthogonal groups are usually denoted $\RepOt$. They are the additive and pseudo-abelian completion of the (unoriented) Brauer categories. In a certain sense, they also interpolate the representation categories of all orthosymplectic Lie supergroups (see \Cref{sec::RepOt}).

We suggest to use interpolation categories as a source of uniform bases in the morphism spaces for a given family of groups.
For the case of the orthosymplectic groups and their fundamental representations, these bases consist of (suitable symmetrizations of) arc diagrams, which we parametrize, in turn, using a combinatorial datum which can be represented using pseudographs and partitions.
We then show that the Jacobi identity characterizing certain PBW deformations can be translated into a graphical calculus on the level of the pseudographs and partitions.
Eventually, we solve the transformed Jacobi identity to obtain a classification of all PBW deformations in question in terms of the aforementioned uniform bases for all orthosymplectic groups.

More concretely, we consider an interpolation version $V$ of the standard representation of the orthosymplectic groups. $V$ is an object in Deligne's interpolation category $\RepOt$. The exterior powers $\L^e V$ in $\RepOt$ for $e\ge1$ can then be regarded as interpolation versions of the fundamental representations. We classify PBW deformations of the smash product algebras $S\L^e V\rtimes\Ug$ and $\L\L^e V\rtimes\Ug$, where $\Ug$ is the interpolation version of the universal enveloping algebra of the Lie algebras of the orthosymplectic groups, and where $S\L^eV$ and $\L\L^eV$ are the symmetric and the exterior algebra of the fundamental representations described above. We call these PBW deformations of type $(\Ug,\L^eV,\rho)$, where we set $\rho=1$ for the symmetric algebra and $\rho=-1$ for the exterior algebra.

\begin{theorem*}[\Cref{thm::PBW-defs}] A basis of the space of interpolating PBW deformations of type $(\Ug,\L^e V,\rho)$ is given by the following explicitly constructed PBW deformation maps: 
\begin{itemize}
    \item a family of deformation maps parametrized by the natural numbers for $e=1$,
    \item a symmetric form on $\L^2 V$ if $e=2$ and $\rho=1$,
    \item a Lie bracket on $\L^2 V$ if $e=2$ and $\rho=-1$,
    \item a family of deformation maps parametrized by pairs of partitions $(\mu,\nu)$ for even $e\ge 2$.
\end{itemize}
\end{theorem*}

We show how our classification for $e=1$, that is, for the standard representation, recovers the known PBW deformations for symplectic and orthogonal groups over the complex numbers from \cite{EGG} and \cite{Tsy}. The techniques used in these references are geometric and completely different from our solution, which uses interpolation categories and is eventually combinatorial in nature. In particular, we provide a new combinatorial formula for these known deformations.

By construction, our solution is uniform across the entire family of orthosymplectic groups and across all exterior powers of the natural representation, which play the role of fundamental representations.

We believe that our results will open up new ways to classify PBW deformations, including those related to other families of Lie groups or finite groups and including those in positive characteristics, which we plan to demonstrate in follow-up projects. Beyond the classification problem, we hope that the techniques we introduce are suitable to study the representations of classes of PBW deformations, for instance, to determine their centers or to describe their irreducible representations.

\textbf{Acknowledgements.} 
We thank Lars G\"ottgens for fruitful discussions and his implementations of some ideas of this paper for the software project \texttt{PBWDeformations.jl} (\cite{PBWDeformations.jl}) together with J.~F. 
The research of J.~F.~ was funded by the Deutsche Forschungsgemeinschaft (DFG, German Research Foundation) -- Project-ID 286237555 -- TRR 195. 
This article is part of the PhD thesis of V.~M.

\setcounter{tocdepth}{1}
\tableofcontents

\section{Background}

Throughout, $\kk$ will be a field of characteristic $0$.

We recall some concepts we will be using, and fix some notations.

\subsection{Partitions and pseudographs} \label{sec::graphs-partitions}
For us, a \emph{partition} is a sequence $\l$ of positive integers $\l_1,\l_2,\dots$ called \emph{parts} such that $\l_1\ge\l_2\ge\dots$ and $\l_i\neq0$ for only finitely many $i\ge1$. We denote by $\#\l$ the smallest $i\ge0$ such that $\l_i=0$, and by $|\l_i|$ the sum of all parts $\sum_{i\ge0}\l_i$. For any positive integer $\ell$, $m_\l(\ell)$ shall denote the multiplicity of $\ell$ as a part in $\l$. We denote by $P$ the set of all partitions.

For us, a \emph{pseudograph} is a graph with finitely many numbered vertices, with undirected edges without identity, where parallel edges between the same vertices and loops are allowed. For any pseudograph $\g$, we denote by $\#\g$ the total number of edges. We will consider $\NN$-labelled pseudographs, where a non-negative number is assigned to each edge as its label. For an $\NN$-labelled pseudograph $\g$, we denote by $|\g|$ the sum of all edge labels, and for any labelled edge $\g_i$ of $\g$ we denote by $m_\g(\g_i)$ its multiplicity in the graph $\g$. Usually, pseudographs will be of constant valence, that is, the number of ingoing or outgoing edges at each vertex is constant, where loops are counted twice.

\subsection{Deligne's interpolation category $\RepOt$} \label{sec::RepOt}
For any $t\in\kk$, we consider the following $\kk$-linear symmetric monoidal category $\cC_t$, the \emph{Brauer category} with parameter $t$: the objects of $\cC_t$ are denoted by $V^{\o k}$ for $k\ge0$. The hom-space $\cC_t(V^{\o k},V^{\o\ell})$ between objects $V^{\o k}, V^{\o \ell}$ for $k,\ell\ge0$ is given by the free $\kk$-vector space spanned by the set of arc diagrams $A^k_\ell$ with $k$ upper and $\ell$ lower points. Here, an arc diagram is a diagram with a given number of upper and lower points, and a set of strings (``arcs'') connecting each point to exactly one other point. The composition in $\cC_t$ is given by a family of operations $A^k_\ell\times A^\ell_m\to\cC_t(V^{\o k},V^{\o m})$ for $k,\ell,m\ge0$; for each pair of arc diagrams, we produce a new arc diagram by stacking them vertically and identifying the $\ell$ common points. In the resulting diagram, we remove all closed loops which consist of arcs among the $\ell$ common points only, and denote the number of removed loops by $L\ge0$. The result of the operation is the arc diagram without loops obtained in this way, multiplied by a scalar factor $t^L$. This is, where the parameter $t$ is used. We make $\cC_t$ a strict monoidal category by setting $V^{\o k}\o V^{\o\ell}:=V^{\o(k+\ell)}$ and by defining the tensor product of two arc diagrams as their horizontal concatenation. The symmetric braiding $c$ in $\cC_t$ is uniquely determined by its component $c_{V,V}$, which is the arc diagram $\cross$ with two upper and lower points.

The following are examples of compositions and tensor products of two arc diagrams: 
$$
\tp{1,3,0,2,12,0,4,13}\circ
\tp{1,12,0,2,14,0,3,4,0,11,13}
=t\,\tp{1,12,0,2,13,0,3,4}
,\qquad
\tp{1,2}\o
\tp{1,12,0,2,11}
= \tp{1,2,0,3,13,0,4,12}
 .
$$

Deligne's interpolation category $\RepOt$ is now defined as the pseudo-abelian additive completion of $\cC_t$, that is, we add formal direct sums and images of idempotents. Hence, any object in $\RepOt$ can be written as $\im(e)$, for an idempotent endomorphism $e$ of an object $X\in\RepOt$ which is a direct sum of objects from $\cC_t$, and $\im(e)$ is a direct summand of $X$ in $\RepOt$. $\RepOt$ is a $\kk$-linear symmetric monoidal whose structure is completely determined by the Brauer (sub)category $\cC_t$.

Deligne showed in \cite{Del}*{Prop.~9.4} that $\RepOt$ is the universal pseudo-abelian tensor category with an object of dimension $t$ which has a symmetric self-duality (``autodualité''). More precisely, this is a structure consisting of a triple $(V',f_\cup\,f_\cap)$, where $V'$ is an object and $f_\cup\:V'\o V'\to\one$, $f_\cap\:\one\to V'\o V'$ are morphisms in a pseudo-abelian tensor category $\cC$ with tensor unit $\one$ such that
\begin{equation} \label{eq::universal-property}
(V'\o f_\cup)(f_\cap\o V') = V' = (f_\cap\o V')(V'\o f_\cap) ,\quad
f_\cup = f_\cup c_{V',V'} ,\quad
f_\cap = c_{V',V'} f_\cap ,\quad
f_\cup f_\cap = t ,
\end{equation}
with $c_{V',V'}$ the braiding in $\cC$. It is an easy check to see that $(V,\cup,\cap)$ is indeed such a structure, with $V\in\RepOt$ as above and with $\cup$ and $\cap$ the respective arc diagrams with one arc and with two upper or lower points viewed as morphisms in $\RepOt$. Given any pseudo-abelian tensor category $\cC$, the symmetric monoidal functors from $\RepOt$ to $\cC$ are given exactly by the triples $(V',f_\cup,f_\cap)$ satisfying \Cref{eq::universal-property}.

Let $\SVec_\kk$ be the category of supervector spaces. By the universal property of $\RepOt$, for any $m,n\ge0$ with $t=m-2n$, there is a unique symmetric monoidal specialization functor $F_{m,2n}\:\RepOt\to\SVec_\kk$ which sends $V:=V^{\o 1}$ to the supervector space $V_\kk:=\kk^{m|2n}$, the arc diagram $|$ with one upper and one lower point to the identity on $V_\kk$, and the arc diagrams $\cup$ and $\cap$ to suitable evaluation and coevaluation morphism of $V_\kk$, respectively (this will be described more explicitly in \Cref{sec::specialization}). Viewing $V_\kk$ as the natural module for the orthosymplectic group $\OSp(m|2n)$, we can even consider $F_{m,2n}$ as a functor to the tensor subcategory of $\Rep\OSp(m|2n)$ which is generated by $V_\kk$. This functor is full (\cite{LZ}*{Cor.~5.8}, \cite{CH}*{Thm.~7.3(i)}).

\subsection{Symmetrization morphisms} \label{sec::symmetrization}
For any $V$ object in any $\kk$-linear symmetric monoidal category $\cC$ and any $n\ge0$, we use the notation $T^n V:=V^{\o n}$. Then we have a group homomorphism $S_n\to\Aut(T^n V)$, where the transposition $(i,i+1)\in S_n$ is sent to the automorphism given by the corresponding braiding 
$$ T^{i-1} V\o c_{V,V}\o T^{n-i-1}V \tforall 1\le i<n .
$$
We will often identify permutations in $S_n$ with their images under such a homomorphism. For any permutation $\s\in S_n$ and object $V$ as above, we denote by $\s^{(V)}$ the automorphism of $T^n V$ defined in this way.
We will also use the endomorphisms
$$
\sym_{\pm n}^{(V)} := \sum_{\s\in S_n} (\pm 1)^\s \s^{(V)} 
$$
of $V$, where $1^\s=1$ and $(-1)^\s$ is the sign of the permutation $\s$.

Whenever they exist, we denote the images of the endomorphisms just defined by
$$
S^{\pm n} V := \im(\sym_{\pm n}^{(V)}) .
$$
This is the case, for instance, if $\cC$ is abelian or, more generally, if $\cC$ is pseudo-abelian, because $\sym_{\pm n}^{(V)}/n!$ is an idempotent endomorphism.

For clarity and brevity, let us write $\L^n V:=S^{-n} V$ where appropriate.

We specialize these general definitions to $\cC=\RepOt$: for any transposition $\tau=(i,i+1)$ in $S_n$, for $1\le i<n$, the corresponding automorphism $\tau^{(V)}$ is given by the arc diagram
$$
 |^{i-1} \,\cross\, |^{n-i-1} .
$$
For any $\s\in S_n$, we denote by $\s^{(m)}$ the permutation in $S_{mn}$ which sends $mk-m+i$ to $m\s(k)-m+i$ for all $0\le k<n$ and $1\le i\le m$, that is, the permutation which permutes all $i$-th points in the tuples $(1,\dots,n),\dots,(mn-n+1,\dots,mn)$ synchronously. Interpreted as an automorphism of the object $T^m V \in\RepOt$, this is the same as $\s^{(T^m V)}$. 
As a consequence, the element
$$
\sym_{\pm n}^{(m)} := \sum_{\s\in S_n} (\pm 1)^\s \s^{(m)}
$$
in the group algebra $\kk S_n$, interpreted as an endomorphism of $T^m V$, is the same as $\sym_{\pm n}^{(T^m V)}$. For $m=1$, we may omit the superscript index, and write $\sym_{\pm n}$ for the endomorphism of the generating object $V$ of $\RepOt$. Let us also write $\alt_n:=\sym_{-n}$.

Finally, for any symmetric monoidal functor $F\:\RepOt\to\cC$, we have
$$
S^{\pm n} T^m F(V) = \im(\sym_{\pm n}^{(T^m F(V))}) \cong \im(F(\sym^{(m)}_{\pm n})) .
$$

\subsection{Bialgebras and smash products} Let $\cC$ be any symmetric monoidal category with tensor unit $\one$ and symmetric braiding $c$. An \emph{algebra} in $\cC$ is a triple $(A,\eta\:\one\to A,\mu\:A\o A\to A)$ such that
$$
\mu(H\o\eta)=\id_H=\mu(\eta\o H)
 , \quad
\mu(\mu\o H)=\mu(H\o\mu)
 .
$$
A \emph{bialgebra} in $\cC$ is a tuple $(H,\eta\:\one\to H,\mu\:H\o H\to H,\varepsilon\:H\to\one,\Delta\:H\to H\o H)$ such that $(H,\eta,\mu)$ is an algebra and, furthermore,
$$
(H\o\varepsilon)\Delta=\id_H=(\varepsilon\o H)\o\Delta
 , \quad
(\Delta\o H)\Delta=(H\o\Delta)\Delta
 ,
 $$
$$
\varepsilon\eta=\id_\one
 , \quad
\varepsilon\mu=\mu(\varepsilon\o\varepsilon)
 , \quad
\Delta\eta=\eta\o\eta
 , \quad
\Delta\mu=(\mu\o\mu)(H\o c_{H,H}\o H)(\Delta\o\Delta) .
$$
Here, the identities in the first line mean that $(H,\varepsilon,\Delta)$ is a coalgebra, while
the identities in the second line mean that the coalgebra structure maps are morphisms of algebras, and vice versa. We call $\eta$ \emph{unit}, $\mu$ \emph{multiplication}, $\varepsilon$ \emph{counit}, and $\Delta$ \emph{comultiplication}. $A$ or $H$ as above is called \emph{commutative} if $\mu=\mu c$, and $H$ is called \emph{cocommutative} if $\Delta=c\Delta$.

For any bialgebra $H$ in $\cC$, the hom-space $\Hom_\cC(H,H)$ is a monoid with identity element $\eta\varepsilon$, with respect to the operation
$$
f*g=\mu(f\o g)\Delta .
$$
$H$ is called a \emph{Hopf algebra} if $\id_H$ has an inverse $S$, called the \emph{antipode}, in this monoid.

Given a bialgebra $H$ in $\cC$, an \emph{$H$-module} is an object $V\in\cC$ together with a morphism $\phi\:H\o V\to V$ such that
$$
\phi(\eta\o V)=\id_V
 , \quad
\phi(H\o\phi)=\phi(\mu\o V)
 .
$$
In this case, for each $i\ge0$, the tensor power $T^iV$ is an $H$-module with a structure which can be defined recursively by $\phi_{T^0V}=\phi_\one=\varepsilon$, $\phi_{T^1 V}:=\phi$, and
$$
\phi_{T^{i+1} V} = (\phi_{T^iV}\o\phi)  (H\o c_{H,T^i V}\o V) (\Delta\o T^{i+1} V) .
$$
If $\cC$ is $\kk$-linear, then by restriction, the symmetric power $S^iV$ and the exterior power $S^{-i}V$ are $H$-modules, as well, if they exist. Finally, the tensor algebra, $TV:=\bigoplus_{i\ge0}T^iV$, the symmetric algebra $SV:=\bigoplus_{i\ge0}S^iV$, and the exterior algebra $\L V:=\bigoplus_{i\ge0}S^{-i}V$ are $H$-modules.

Now given a bialgebra $H$ and an $H$-module $V$, we define
$$
\ti c_{H,TV} := (\phi_{TV}\o H)(H\o c_{H,TV})(\Delta\o TV) \:  H\o TV\to TV\o H
 .
$$
Then $TV\o H$ is an algebra with
$$
\eta:=\eta_{TV}\o\eta_H
\ , \quad
\mu:=(\mu_{TV}\o\mu_{H})(TV\o\ti c_{H,TV}\o H)
 .
$$
This algebra is called the \emph{smash product} of $TV$ and $H$ and denoted $TV\rtimes H$. Similarly, we define the smash products $SV\rtimes H$ and $\L V\rtimes H$.

\subsection{PBW deformations} \label{sec::PBW}
A filtered algebra $A$ in a $\kk$-linear abelian symmetric monoidal category $\cC$ is called a \emph{PBW deformation} of a given graded algebra $A_0$ in $\cC$, if the associated graded algebra $\gr(A)$ of $A$ is isomorphic to $A_0$.

As a special case, the smash product $A:=TV\rtimes H$, for any bimodule $H$ and any $H$-module $V$ in $\cC$, is filtered with $F^i(TV\rtimes H):=\bigoplus_{0\le j\le i} T^i V\o H$. Then for any $\rho\in\{\pm1\}$, $\k\:S^{2\rho}V\to H$, let $I_\k$ be the (two-sided) ideal in $A$ generated by the image of the morphism $\iota_{S^{2\rho}V}-\k$, where $\iota_{S^{2\rho} V}\:S^{2\rho} V\to T^2 V$ is the natural embedding. Then the quotient algebra $A_\k:=A/I_\k$ is a filtered algebra, and we can ask if it is a PBW deformation of $A_0:=A/I_0$, i.e., the quotient algebra for the case $\k=0$. 

We will call PBW deformations $A_\k$ of $A_0$ as above \emph{PBW deformations of type $(H,V,\rho)$}, and the main goal of this paper will be to classify PBW deformations of certain types.

\section{Parametrizing morphism spaces} \label{sec::parametrizing}

As a first step of our classification of PBW deformations, we will derive a parametrization of bases in certain morphism spaces. Since the involved objects are given by combinations of symmetric or exterior powers, the basis elements with be symmetrizations of arc diagrams with respect to suitable group actions.

\begin{definition} 
For any $p,e,d\ge0$, we define the finite group 
$$
G_{p,e,d}:=S_e^p\times(S_2\wr S_d)=S_e^p\times (S_2^d\rtimes S_d) .
$$
\end{definition}

Recall from \Cref{sec::RepOt} that, for all $k,\ell\ge0$, $A^k_\ell$ denotes the set of arc diagrams with $k$ upper and $\ell$ lower points.

We have an action of the group $G_{p,e,d}$ on $A^{pe}_{2d}$ by permuting points, as follows: Each group $S_e$ acts on a set of $e$ adjacent upper points, grouping them $B_1:=(1,\dots,e)$, $\dots$, $B_p:=((p-1)e+1,\dots,pe)$. Each of the $d$ copies of $S_2$ in the wreath product $S_2\wr S_d$ acts on a pair of adjacent lower points, grouping them $(1,2),\dots,(2d-1,2d)$. Finally, the group $S_d$ in the wreath product permutes these pairs of lower points, preserving the order in each pair.

Let us define a datum which is assigned to any arc diagram $x\in A^{pe}_{2d}$. We say two (distinct) arcs in $x$ are \emph{adjacent} if there is a pair of lower points whose members are end points of one arc each. We group the arcs of $x$ into arc sequences such that any two consecutive arcs are adjacent in this sense. Each arc sequence may or may not contain upper points, but if it does, then the upper points are the start and end points of the entire arc sequence, as all other points must be lower points by construction. Now any sequence can be described by a tuple $(i,j,k)$ where $B_i$ and $B_j$ are the groups of upper points containing the start and end points if $1\le i<j\le e$, or $i=j=0$ if there are no upper points in the sequence, and $2k$ is the number of lower points in the sequence for some $k\ge0$. 

\begin{definition} For any arc diagram $x\in A^{pe}_{2d}$, let $\ASD(x)$ be the set $\{(i,j,k)\}$ describing all arc sequences; we call $\ASD(x)$ the \emph{arc sequence datum}.
\end{definition}

\begin{lemma} \label{lem::ASD}
$\ASD$ is a separating invariant for the action of $G_{p,e,d}$ on $A^{pe}_{2d}$.
\end{lemma}

\begin{proof} It is straight-forward to see that a set of generators of $G_{p,e,d}$ does not change the arc sequence datum, so $\ASD$ is an invariant. 

It remains to show that two arc diagrams $x$ and $x'$ lie in the same orbit if they have the same arc sequence datum. To see this, we identify the points along each arc sequence in $x$ with the points along a sequence of points of the same length in $x'$ such that all upper points in a tuple $B_i$ are identified with upper points within the same tuple $B_i$, and pairs of lower points are identified with pairs of lower points. The permutation representing this identification is in $G_{p,e,d}$.
\end{proof}

Recall from \Cref{sec::graphs-partitions} that a pseudograph is an undirected graph with multiple edges (without identity) and loops. We consider $e$-valent $\NN$-edge-labelled pseudographs for some $e\ge0$, that is, we require the number of edges adjacent to any vertex to be constant (counting loops twice), and to any edge there is a non-negative integer associated. 

\begin{definition} For $p,e\ge0$, let $\ti\Gamma_{p,e}$ be the set of $\NN$-edge-labelled $e$-valent pseudographs with vertices $\{1,\dots,p\}$.
\end{definition}

For every $\g\in\ti\Gamma_{p,e}$, we fix an arbitrary listing $(\g_{i,1},\g_{i,2},|\g_i|)_{1\le i\le ep}\subset\{1,\dots,p\}^2\times\NN$ of the edges with their vertices and their labels. We pick $b_{\g,i,j}\in B_{\g_{i,j}}$ such that $\{b_{\g,i,j}\}_{1\le i\le pe,1\le j\le 2}=\{1,\dots,pe\}$. Then the sum of edge labels $|\g|$ can be expressed as $|\g|=|\g_1|+\dots+|\g_{pe}|$.

\begin{definition} \label{def::x-g-l}
For all $p,e\ge0$, $\g\in\ti\Gamma_{p,e}$, $\l\in P$, we set
$$
x(\g,\l) := 
\bigotimes_{\g_i} |^{(\g'_{i,1})}\cap^{|\g_i|}|^{(\g'_{i,2})}
\o
\bigotimes_{\l_j} \arch{\cap^{\l_j-1}}
\quad\in A^{pe}_{2(|\g|+|\l|)} ,
$$
with the following convention: the order of the tensor product factors is left-to-right, each cap symbol $\cap$ indicates an arc connecting two lower points, each bar symbol $|$ indicates an arc connecting an upper point with a lower point, where the superscript indices in parentheses to the right of any bar $|$ indicate the upper point the corresponding arc is connected to; the lower points for all arcs are given by the order in which appear in the vertical concatenation. 
\end{definition}

In words, $x(\g,\emptyset)$ is an arc diagram with $pe$ upper points and $2|\g|$ lower points. The upper points are split into $p$ groups $B_1,\dots,B_p$ of $e$ adjacent points. To each group we associate a vertex of the graph. Each edge $\{\g_{i,1},\g_{i,2}\}$ with label $|\g_i|$ is turned into a sequence of $|\g_i|+1$ arcs, starting at an upper point in the group $B_{\g_{i,1}}$ ending in an upper point in the group $B_{\g_{i,2}}$, via a total of $2|\g_i|$ intermediate adjacent lower points. The exact start point and end point within the respective group corresponds to a choice of the $(b_{\g,i,j})$. 

As we are interested in arc diagrams only up to the action of $G_{p,e,d}$, this choice and the choice of the listing $(\g_{i,1},\g_{i,2},|\g|)$ for a given graph $\g$ are not relevant for our purposes. However, we want to fix these choices once and for all.

\begin{lemma} \label{lem::choices}
The orbit of $x(\g,\l)\in A^{pe}_{2d}$ under the action of $G_{p,e,d}$ is independent of the choice of the edge listing $(\g_{i,1},\g_{i,2},|\g_i|)$ and of the exact upper points $\{b_{\g,i,j}\}_{i,j}$ for all edges.
\end{lemma}

\begin{proof} Both choices do not affect the arc sequence datum of the arc diagrams constructed.
\end{proof}

\begin{example} For $p=2$, $e=4$, the following represents an element $\g\in\ti\Gamma_{p,e}$ and one of the possible arc diagrams $x(\g,\emptyset)$ associated to it. Note the $p=2$ groups consisting of $e=4$ upper points each, and the $|\gamma|=3$ pairs of lower points.
$$
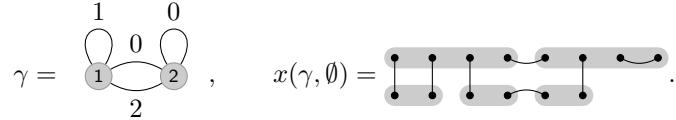
\g=
\tikz[baseline=(a1.base)]{
\mydots{2} \myloop{a1}{1} \myloop{a2}{0}
\draw[-] (a1) to[out=30,in=150] node[midway,above]{0} (a2);
\draw (a1) to[out=-30,in=-150] node[midway,below]{2} (a2);
} 
,\qquad x(\g,\emptyset)=
\tpa{2}{4}{3}{1,11,0,12,2,0,3,13,0,14,15,0,16,6,0,4,5,0,7,8}.
$$
\end{example}

For a group $G$ acting on a set $X$ and an element $x\in X$, we denote by 
$G_x := \{ g\in G: g\cdot x=x \}$ the \emph{stabilizer (sub)group}. Also, let
$D_\ell := \langle a,b: a^2=b^\ell=1, aba=b^{-1} \rangle$
be the dihedral group of order $2\ell$ for any $\ell\ge1$. As is customary, we call $a$ the \emph{reflection} and $b$ the \emph{rotation} in $D_\ell$. For any group $G$ and $n\ge1$, denote by $G\wr S_n$ the \emph{wreath product} group, which can be viewed as a semidirect product $G^n\rtimes S_n$, where the action of $S_n$ on the product group $G^n$ is by permutation of the factors.

\begin{lemma} \label{lem::stabilizer}
The stabilizer group of the arc diagram $x=x(\g,\l)$ with respect to the action of $G=G_{p,e,d}$ on $A^{pe}_{2d}$ is given by
\begin{align*}
G_x 
&\cong
\prod_{1\le i\le j\le e,\ell\ge0}
(S_2^{\delta_{ij}}\wr S_{m_\g((i,j,\ell))})
\times \prod_{\ell\ge0}
(D_{\ell} \wr S_{m_\l(\ell)}) 
\\
&= \prod_{1\le i\le j\le e,\ell\ge0}
(S_2^{\delta_{ij}m_\g((i,j,\ell))}\rtimes S_{m_\g((i,j,\ell))})
\times \prod_{\ell\ge0}
(D_{\ell}^{m_\l(\ell)}\rtimes S_{m_\l(\ell)}) ,
\end{align*}
where the groups $S_{m_\g((i,j,\ell))}$ and $S_{m_\l(\ell)}$ permute arc sequences of a given type, either with or without upper points, the generators in each $S_2$ invert the order of all upper and lower points in an arc sequence with upper points within the same block of $e$ upper points, the reflection in each $D_{\ell}$ reverses the order of all lower point in an arc sequence without upper points, and the rotation in $D_{\ell}$ rotates the $\ell$ pairs of adjacent lower points in an arc sequence without upper points.
\end{lemma}

\begin{proof} We have seen above that the permutations of upper and lower points given by elements of $G$ preserve upper and lower points, adjacent arcs and, hence, also arc sequences. So any element in the stabilizer in question is uniquely determined by how it permutes arc sequences of the same shape, and by how it permutes points within each arc sequence.

Since even within one arc sequence, adjacent arcs are preserved, the only permutations possible within an arc sequence with upper points is the one interchanging the two upper points and reversing all lower points, which is possible only if the two upper points are in the same block of upper points. Again, since adjacent arcs are preserved, the permutations within an arc sequence consisting of only lower points are given exactly by the images of, say, the first two adjacent lower points. These must form another pair of adjacent lower points. The permutation of all points of the considered arc sequence now is a rotation if the order of the two adjacent points is preserved, or a rotation followed by a reflection otherwise.

This yields the asserted description of the stabilizer group.
\end{proof}

We will consider a symmetrization map $f$ for the action of $G_{p,e,d}$ on $A^{pe}_{2d}$ with respect to a (sign) character.

\begin{definition}
For all $p,e,d,\ge0$, we define a character $\chi_{p,e,d}\colon G_{p,e,d}=S_e^p\times (S_2^d\rtimes S_d)\to\kk^\times$ by
$$
\chi_{p,e,d}((\theta,(\tau,\sigma)))
= (-1)^\theta (-1)^\tau
$$
for all $\theta\in S_e^p$, $\tau\in S_2^d$, $\s\in S_d$.
\end{definition}

We will omit the indices $p,e,d$ whenever they are clear from the context.



\begin{definition} Let $P^+\subset P$ be the set of partitions with only even parts, and let $\Gamma_{p,e}\subset\ti\Gamma_{p,e}$ be the subset of those pseudographs whose loops are labelled with odd numbers only. \end{definition}

\begin{proposition} \label{prop::orbit-reps} For any $p,e,d\ge0$, the set 
$$
\{x(\g,\l) : \g\in\Gamma_{p,e},\l\in P^+,|\g|+|\l|=d\}
$$
is a set of representatives $x$ of exactly those orbits of $G_{p,e,d}$ acting on $A^{pe}_{2d}$ for which $\chi|_{(G_{p,e,d})_x}\equiv 1$.
\end{proposition}

\begin{proof} By \Cref{lem::ASD}, the orbit of $x$ is uniquely determined by the arc sequence datum $\ASD(x)$. Now we associate a pseudograph $\g\in\ti\Gamma_{p,e}$ to the set of all arc sequences in $x$ which contain upper points by creating one edge for each sequence between the vertices $i,j$ if the upper points lie in $B_i$ and $B_j$, respectively, and by labelling this edge with the number of pairs of lower points the arc sequence passes through. We associate a partition $\l\in P$ to the set of all arc sequences in $x$ which do not contain upper points, whose parts are the number of lower pairs each such sequence passes through. Then as $x(\g,\l)$ has the same arc sequence datum as $x$, the set $\{x(\g,\l):g\in\ti\Gamma_{p,e},\l\in P\}$ is a complete set of orbit representatives of $G_{p,e,d}$ acting on $A^{pe}_{2d}$.

Now consider $H:=(G_{p,e,d})_x\subset G_{p,e,d}$ for $x=x(\g,\l)$. It follows directly from the description of $H$ given in \Cref{lem::stabilizer}, that a set of generators $g$ of $H$ satisfies $\chi(g)=1$ if and only if $\g\in\G_{p,e}$ and $\l\in P^+$.
\end{proof}

This will imply a result on bases in hom-spaces with the following general auxiliary result.

\begin{lemma} \label{lem::general-group-action} Let $G$ be a finite group acting on a set $S$, let $\chi\: G\to\kk^\times$ be a group homomorphism. Define $f\colon\kk S\to\kk S$, $x\mapsto \sum_{g\in G}\chi(g) (g\cdot x)$. Then
$$
f(x) = 
\begin{cases}
\#G_x \sum_{gG_x\in G/G_x} \chi(g) (g\cdot x) & \chi|_{G_x}\equiv1 \\
0 & \chi|_{G_x}\not\equiv1
\end{cases}
.
$$
In particular, let $\ti X\subset S$ be any set of orbit representatives, and $X:=\{x\in\ti X:\chi|_{G_x}\equiv1 \}$. 
Then $f$ restricts to a linear isomorphism $f|_{\kk X}\colon \kk X\to f(\kk S)$ and $f|_{\ti X\setminus X}\equiv0$.
\end{lemma}

\begin{proof} 
Let $(g_i)_i$ be any complete set of coset representatives such that $G=\bigsqcup_i g_i G_x$. Then
$$
f(x) = \Big( \sum_{g\in G_x} \chi(g) \Big) \sum_i \chi(g_i) (g_i\cdot x).
$$
As
$$
\sum_{g\in G_x} \chi(g) = \chi(h) \sum_{g\in G_x} \chi(g)
$$
for all $h\in G_x$, the sum $\sum_{g\in G_x} \chi(g)$ is zero if $\chi|_{G_x}\not\equiv1$, and is $\# G_x$ if $\chi|_{G_x}\equiv 1$. Note that the formula for $f(x)$ is independent of the choice of coset representatives.
\end{proof}

Recall from \Cref{sec::RepOt} that the arc diagrams $A^k_\ell$ form a basis of the morphism space $\Hom(T^k V,T^\ell V)$ in $\RepOt$. We will now describe bases in objects given by symmetric or exterior products, or combinations thereof.  In the following, $\Hom$ always refers to the category $\RepOt$.

\begin{definition} For all $p,e,d\ge0$ and $G:=G_{p,e,d}$, we define $f_{p,e,d}\:A^{pe}_{2d}\to\Hom(T^{pe}V,T^{2d}V)$ by setting
$$
f_{p,e,d}(x):=
\sum_{g\in G} \chi_{p,e,d}(g)(g\cdot x)
\tforall x\in A^{pe}_{2d} .
$$
\end{definition}

Recall from \Cref{sec::symmetrization} that in our situation, for every $X\in\RepOt$ and $n\ge0$, symmetric and exterior powers exists and are given by
$$
S^{\pm n}X = \im(\sym^{(X)}_{\pm n}) .
$$
As $\sym^{(X)}_{\pm n}$ can be normalized to be an idempotent, we may view these objects as direct summands of the corresponding tensor powers $X^{\o n}$.

\begin{definition} For any $d\ge0$, let us define
$$
s_d := \sym^{(\L^2 V)}_d = \sym^{(2)}_{d}\circ\sym_{-2}^{\o d}=\sym_{-2}^{\o d}\circ\sym^{(2)}_{d} ,
$$
the endomorphism defining the object $S^d\L^2 V$ in $\RepOt$.
\end{definition}

\begin{corollary} \label{cor::Hom-tensor}
For all $p,e,d\ge0$, the set
$$
\{ s_d\circ x(\g,\l)|_{T^p\L^e V} :\g\in\Gamma_{p,e},\l\in P^+,|\g|+|\l|=d\}
$$
is a basis of $\Hom(T^p\L^e V,S^d\L^2V)$.
\end{corollary}

\begin{proof} Note that 
\begin{equation} \label{eq::formula-f}
s_d\circ x\circ\alt_e^{\o p}
= \sum_{g\in G_{p,e,d}} \chi_{p,e,d}(g)(g\cdot x)
= f_{p,e,d}(x)
\tforall x\in A^{pe}_{2d} .
\end{equation}
Hence, the assertion follows from \Cref{prop::orbit-reps} and \Cref{lem::general-group-action}.
\end{proof}

To understand, why not all of $\ti\G_{p,e}$ and not all of $P$ are needed to parameterize a basis, consider these examples:

\begin{example} Let $\g\in\ti\Gamma_{1,2}$ be a pseudograph with one vertex and a loop with even label $d\ge 0$. Consider $x=x(\g,\emptyset)$ in $A^2_{2d}$. Let $\pi$ be the element in $G_{1,2,d}$ which permutes the two upper points and inverts the order of all $2d$ lower points in any arc diagram of $A^2_{2d}$. Then, as $\pi\cdot x(\g,\emptyset)=x(\g,\emptyset)$ and $\chi(\pi)=-1$, 
$$
s_d\circ x\circ\alt_2 = 0 .
$$
The arc diagrams $x(\g,\emptyset)$ and the action of the permutation $\pi$ can be visualized for $d=0,2,4$ as follows:
$$
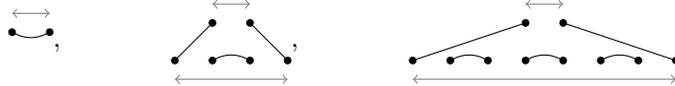
\tpsyma{1}{0}{1,2},
\qquad\qquad
\tpsyma{2}{4}{2,11,0,12,13,0,14,3},
\qquad\qquad
\tpsyma{4}{8}{4,11,0,12,13,0,14,15,0,16,17,0,18,5} .
$$
\end{example}

\begin{example} Similarly, let $\l\in P$ be the partition $\l=(d)$ with one part of odd size $d\ge 1$, let $\pi$ be the element of $G_{0,0,d}$ which permutes the first two lower points and inverts the order of the remaining $2d-2$ lower points in any arc diagram of $A^0_{2d}$. Then $\pi\cdot x(\emptyset,\l)=x(\emptyset,\l)$ and $\chi(\pi)=-1$, so $s_d\circ x(\emptyset,\l)=0$.
The arc diagrams $x(\emptyset,\l)$ and the action of the permutation $\pi$ can be visualized for $d=1,3,5$ as follows:
$$
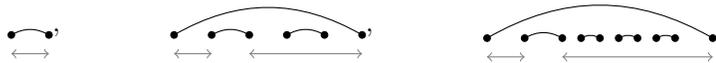
\tpsymb{0}{11,12},
\qquad\qquad
\tpsymb{4}{11,16,0,12,13,0,14,15},
\qquad\qquad
\tpsymb{5}{11,17,0,12,13,0,13.5,14,0,14.5,15,0,15.5,16} .
$$
\end{example}

\bigskip

Note that apart from the action of $G_{p,e,d}$ considered so far, we also have an action of $S_p$ on $A^{pe}_{2d}$ by permuting the $e$-tuples $B_1,\dots,B_p$ of upper points, and we have an action of $S_p$ on $\Gamma_{p,e}$ permuting the vertices, whose orbits correspond to isomorphisms classes of pseudographs, ignoring the vertex numbering.

For the rest of this section, we specialize $p=2$, and we set $\tau:=(1,2)\in S_2$, in order to describe a basis for the morphism spaces $\Hom(S^{\pm2}\L^e V,S^d\L^2 V)$.

\begin{definition} For any $\g\in\G_{2,e}$, we define
$$
\sgn\g := \prod_{\g_i: \g_{i,1}=1, \g_{i,2}=2} (-1)^{|\g_i|} 
\qquad\in\{\pm1\}
$$
and
$$
\g^t:=\tau\cdot\g \in\G_{2,e} ,
$$
the graph with the two vertices interchanged.
\end{definition}

\begin{lemma} \label{lem::sign}
For all $e\ge0$, $\g\in\Gamma_{2,e}$, and $\l\in P^+$,
$$
f_{2,e,|\g|+|\l|}(x(\g,\l))\circ \tau^{(e)} 
= \sgn(\gamma) f_{2,e,|\g|+|\l|}(x(\g^t,\l)) .
$$
\end{lemma}

\begin{proof} The assertion follows immediately with \Cref{eq::formula-f} from the claim that
$$
x(\g,\l)\circ\tau^{(e)} = g\cdot x(\tau\cdot\g,\l)
$$
for some $g\in G_{2,e,d}$ with $\chi_{2,e,d}(g)=\sgn\g$.

To prove the claim, note that as the part of $x(\g,\l)$ which depends on $\l$ does not have arcs ending in upper points, we may assume $\l=\emptyset$. We are free to choose a listing of the edges of $\g$, so we may assume the listing first lists all loops involving the vertex $1$, then all edges between $1,2$, and finally, all loops involving the vertex $2$. We set $x=x(\g,\emptyset)$, using this choice. Let $L_1$ and $L_2$ be the numbers of loops at the vertices $1$ and $2$, respectively. Let $N$ be the total number of edges.

We can represent
$$
x = \prod_{\g_i}
|^{(\g_{i,1},p_{i,1})}\cap^{|\g_i|}|^{(\g_{i,2},p_{i,2})}
$$
where the indices in parentheses indicate the position $p_{i,j}$ of the upper point the respective arc ends in, within the block $B_{\g_{i,j}}$ of upper points. By our choices, $(\g_{i,1},\g_{i,2})=(1,1)$ for $1\le i\le L_1$, $(\g_{i,1},\g_{i,2})=(1,2)$ for $L_1< i\le N-L_2$, and $(\g_{i,1},\g_{i,2})=(2,2)$ for $N-L_2< i\le N$. We can now represent
$$
x\circ\tau^{(e)}
= \prod_{\g_i}
|^{(\tau(\g_{i,1}),p_{i,1})}\cap^{|\g_i|}|^{(\tau(\g_{i,2}),p_{i,2})} .
$$
Let $g\in G_{2,e,d}$ be the permutation which inverts the order of all lower points along every arc sequence between $B_1$ and $B_2$. Then $\chi_{2,e,d}(g)=\sgn\g$ and $g\cdot(x\circ\tau^{(e)})$ is one of the possible arc diagrams for $x(\g^t,\emptyset)$, namely for a choice of the listing of all edges analogous to the one chosen for $x=x(\g,\emptyset)$.
\end{proof}

Let $\Gamma_{2,e}/S_2$ be the set of isomorphism classes $[\gamma]$ of pseudographs in $\G_{2,e}$ as $\NN$-edge-labelled pseudographs without a vertex numbering. We note that the function $\g\mapsto\sgn\g$ induces a well-defined function $\Gamma_{2,e}/S_2\ni[\g]\to\sgn[\g]\in\{\pm1\}$.

\begin{definition} \label{def::x-rho}
For $\rho=\pm1$, let $\G^\rho_{2,e}$ be a set of representatives for those isomorphism classes $[\g]\in\G_{2,e}/S_2$ which satisfy (a) $|[\g]|=2$, i.e., $\g\neq\g^t$, or (b) $\sgn[\gamma]=\rho$. For all $\g\in\G^\rho_{2,e}$, $\l\in P^+$, we set
$$
x^\rho(\g,\l) := x(\g,\l) + \rho \sgn(\g) x(\g^t,\l)
\qquad\in\Hom(T^{2e}V,T^{2d}V) .
$$
\end{definition}

\begin{proposition} \label{prop::Hom-sym} 
For all $e,d\ge0$, the set 
$$
\{ s_d\circ x^\pm(\g,\l)|_{S^{\pm2}\L^e V} :\g\in\Gamma^\pm_{2,e},\l\in P^+,|\g|+|\l|=d\}
$$
is a basis of $\Hom(S^{\pm2}\L^e V,S^d\L^2V)$.
\end{proposition}

\begin{proof}
By \Cref{cor::Hom-tensor}, the elements $s_d\circ x(\g,\l)\circ (\id\pm\tau^{(e)})$ for $\g\in\G_{2,e}$ and $\l\in P^+$ span the $\Hom$-space in question. By \Cref{lem::sign}, the asserted set is a basis.
\end{proof}

\begin{remark} \Cref{cor::Hom-tensor} and \Cref{prop::Hom-sym} can be summarized by saying that the sets 
$$\G_{p,e}\x P^+ \tand \G^{\pm}_{2,e}\x P^+ $$
parametrize bases of 
$$ \Hom(T^p \L^e V,S(\L^2 V))
\tand
\Hom(S^{\pm2}\L^e V,S(\L^2 V))
$$
in the ind-completion of $\RepOt$, respectively, where $S(\L^2V)=\bigoplus_{d\ge0} S^d\L^e V$ is an ind-object.
\end{remark}

\section{Interpolating the Jacobi identity}

We will use the idea that $\RepOt$ in a certain precise sense interpolates the representation categories of all orthosymplectic groups (see \Cref{sec::RepOt}) to define interpolating versions of PBW deformations.

To this end, we first explain an interpolating version of the Lie algebras of these groups. As before, the symbol $\Hom$ without further specification refers to the hom-spaces in the category $\RepOt$ or in its ind-completion.

\begin{definition} We set
\begin{align*}
\gf &:= \L^2 V \subset T^2 V
\qquad\in\RepOt
\\
B &:= (\alt_2\circ |\cup|)|_{\gf\o\gf} \qquad\in\Hom(\gf\o\g,\gf)
\\
\phi &:= (|\cup)|_{\gf\o V} 
\qquad\in\Hom(\gf\o V,V)
\\
\Ug &:= S\gf = \bigoplus_{d\ge 0} S^d \gf = \one\oplus V\oplus S^2 V\oplus\dots
\qquad\in\Ind(\RepOt)
\\
\Delta_{i,j} &:= (T^i V\o T^j V)|_{S^{i+j}\gf} \qquad\in\Hom(S^{i+j}\g,S^i\gf\o S^j\gf) \tforall i,j\ge0 
\\
\Delta &:= \bigoplus_{d\ge0} \sum_{i+j=d} \Delta_{i,j}
\qquad\in\Hom(\Ug,\Ug\o\Ug) .
\end{align*}
\end{definition}

\begin{lemma} \label{lem::interpolating-lie-algebra}
In $\RepOt$, $(\gf,B)$ is a Lie algebra, $V$ is a $\gf$-module with the action map $\phi$.
In $\Ind(\RepOt)$, $\Ug$ has the structure of a Hopf algebra, the universal enveloping algebra of $\gf$, with cocommutative comultiplication $\Delta$. 
The specialization functor $F_{m,2n}$ from \Cref{sec::RepOt} sends $\gf$ to the Lie algebra of $\OSp(m|2n)$ and $\Ug$ to its universal enveloping algebra.
\end{lemma}

\begin{proof} The assertions can be checked in a straight-forward way using the categorical versions of Lie algebras and Hopf algebras (see also \cite{E16}*{Prop.~2.5, Cor.~2.14} and \cite{HU}*{Sec.~2.3}).
\end{proof}

We fix an arbitrary $e\ge1$ and set
$$
W := \L^eV .
$$
As $\Ug$ is a cocommutative Hopf algebra in a symmetric monoidal category, it has a natural action on $W=\L^e V$. We postpone the actual computation of this action to \Cref{lem::action-phi-e}, since it will suffice for now to work with some of its general properties.

\begin{definition} Let $\hat\phi\in\Hom(\Ug\o W)$ be the action of $\Ug$ on $W$ and for all $d\ge0$, let $\phi_d\in\Hom(S^d\gf\o W,W)$ be its restriction to $S^d\gf\subset\Ug$.
\end{definition}

We now also fix a $\rho\in\{\pm1\}$, which will allow us to discuss symmetric and exterior products of $W$ simultaneously.

\begin{definition} For any $\k\in\Hom(S^{2\rho}W,\Ug)$, we set
$$
J(\k) := J^\rho(\k) :=
\big(
({\Ug}\o\hat\phi)\circ(\Delta\circ\kappa\o W)
- \kappa\o W
\big)\big|_{S^{3\rho}W}
\quad\in\Hom(S^{3\rho}W, \Ug\o W) .
$$
\end{definition}

Recall from \Cref{sec::PBW} the notion of PBW deformations and PBW deformation maps.

\begin{lemma} For any tensor functor $F\:\Ind(\RepOt)\to\SVec$ and for any $\k\in\Hom(S^{2\rho}W,\Ug)$, $F(\k)$ is a PBW deformation map of type $(F(\Ug),F(W),\rho)$ if and only if $F(J(\k))=0$.
\end{lemma}

\begin{proof} By definition, $F(\k)$ is a PBW deformation map of type $(F(\Ug),F(W),\rho)$ if and only if
$$
A_{F(\k)}:=TF(W)\rtimes F(\Ug) / (vw + \rho wv - F(\k)(vw + \rho wv) : v,w\in TF(W) )
$$
is a PBW deformation of $A_0:=S^{\rho} F(W)\rtimes F(\Ug)$, where $S^+$ and $S^-$ shall denote the symmetric and exterior algebra, respectively.

By the general theory of PBW deformation in the category of supervector spaces in characteristic $0$, this equivalent to two conditions: (a) $F(\k)$ needs to be $F(\Ug)$-equivariant, which is automatic in our case, since $F(\k)$ is the specialization of a morphism from (the ind-completion of) $\RepOt$. (b) the (generalized) Jacobi identity
\begin{align*}
0 &= [[u,v],w] + [[v,w],u] + [[w,u],v]
\end{align*}
need to hold in $A_{F(\k)}$, where $F(\k)$ is used to compute the inner commutators. The equivalence of the conditions (a), (b) with $A_\k$ being a PBW deformation of $A_0$ follow from a categorical version of the ``Koszul deformation principle'' \cite{E18}*{Thm.~3.4}; for the category of vector spaces, this is a special case of \cite{WW}*{Thm.~3.1}.

The assertion now follows upon observing that the nested commutators can be computed as
$$
[[u,v],w] = [F(\k)(uv+\rho vu),w]
= F( (\Ug\o \hat\phi)\o(\Delta\circ\k\o W) - \k\o W)(u\o v\o w) .
$$
\end{proof}

This motivates the following definition:
\begin{definition} $\k\in\Hom(W,\Ug)$ is called \emph{interpolating PBW deformation map of type $(\Ug,W,\rho)$} if $J(\k)=0$.
\end{definition}

We will eventually determine all interpolating PBW deformation maps of type $(\Ug,W,\rho)$. We start with some reformulations of the defining condition $J(\k)=0$.

For all $d\ge0$, let $\pi_d\in\Hom(\Ug, S^d\gf)$ be the projection onto the degree-$d$ part of $\Ug$, and set $\k_d:=\pi_d\circ\k$, so 
$$
\k = \sum_{d\ge0} \k_d .
$$

We can now formulate a first reduced version of the Jacobi identity for $\k$ in the following sense:

\begin{lemma} \label{lem::Jk-1}
$J(\k)=0$ if and only if 
$$
(|^{2d-2}\o\phi_1)\circ(\k_d\o W)|_{S^{3\rho}W}=0 
\tforall d\ge 1.
$$
\end{lemma}

\begin{proof} By the definition of $\D$ and $\hat\phi$, we have for all $d\ge0$,
$$
(\Ug\o\hat\phi)\circ(\Delta\o W)|_{{S^d\gf\o W}}
= \sum_{i=0}^d (|^{2d-2i}\o\phi_i)|_{S^d\gf\o W}
= S^d\gf\o W + \sum_{i=1}^d (|^{2d-2i}\o\phi_i)|_{S^d\gf\o W} .
$$
Hence,
\begin{align*}
J(\k) &= 
\sum_{d\ge0} 
\Big(
(\Ug\o\hat\phi)\circ(\Delta\circ\kappa_d\o W)
- \kappa_d\o W
\Big)|_{S^{3\rho}W}
\\
&=  \sum_{1\le i\le d} (|^{2d-2i}\o\phi_i)\circ(\k_d\o W)|_{S^{3\rho}W} .
\end{align*}
Set 
$$ J_{i,d}:=(|^{2d-2i}\o\phi_i)\circ(\k_d\o W)|_{S^{3\rho}W}
\quad\in \Hom(S^{3\rho}W,S^{d-i}\g\o W)
\tforall 1\le i\le d.
$$
As $(\phi_i)_i$ are action maps, $|^{2d-2i-2}\o\phi_{i+1}$ factors through $|^{2d-2i}\o\phi_i$ for all $i\ge1$, so $J_{i,d}=0$ implies $J_{i+1,d}=0$ for all $1\le i<d$. This immediately implies the ``if''-part of the assertion.

Assume $J(\k)=0$ and $J_{1,d'}\neq0$
for some $d'\ge0$. Choosing $d'$ maximal, we may assume $J_{1,d}=0$ for all $d>d'$. Hence also $J_{i,d}=0$ for all $d>d'$ and $i\ge1$. So
$$
0
= (\pi_{d'-1}\o W)\circ J(\k)
= \sum_{1\le i\le d: d-i=d'-1} J_{i,d}
= J_{1,d'} ,
$$
a contradiction proving the ``only if''-part.
\end{proof}

Note that there is no condition on $\k_0$ for $J(\k)$ to be zero.

Let us use the notation $c_{k,\ell}:=c_{T^k V,T^\ell V}$ for the braiding in $\RepOt$. Recall from \Cref{sec::RepOt} that it is given by the permutation in $S_{k+\ell}$ which sends $i\mapsto i+\ell$ for $1\le i\le k$ and $i\mapsto i-\ell$ for $k<i\le k+\ell$.

\begin{definition} \label{def::cyc}
We define a distinguished morphism
$$
\cyc := \cyc_e := \sum_{i=1}^e (-1)^i c_{1,i-1}\o|^{e-i}
\quad\in\Hom(T^e V,T^e V) .
$$
\end{definition}

\begin{lemma} \label{lem::action-phi-e}
The action of $\gf$ on $W=\L^e V$ is given by:
\begin{align}
\phi_1 
&= \cyc\circ(|\cup|^{e-1})|_{\gf\o W}
&\quad\in\bigoplus_{d\ge0}\Hom(\gf\o W,W) .
\end{align}
\end{lemma}

\begin{proof} The action of $\gf$ on $\L^e V$ is the restriction of the action of $\gf$ on $T^e V$, which is determined by the (categorical) derivation rule and by $\phi=|\cup$ to be
$$
\sum_{i=1}^e (|^{i-1}\o\phi\o |^{e-i}) \circ (c_{2,i-1}\o |^{e+1-i})
=
\sum_{i=1}^e (|^i \cup |^{e-i}) \circ (c_{2,i-1}\o |^{e+1-i}) ,
$$
where $c$ is the symmetric braiding of the category $\RepOt$, which is given by arc diagrams corresponding to permutations. Comparing arc diagrams, we can identify the summands as
$$
(c_{1,i-1}\o |^{e-i})
\circ (|\cup |^{e-1}) 
\circ (|^2\o c_{i-1,1}\o |^{e-i}) .
$$
The assertion follows, as $W=\im\alt_e$ and
$$
(c_{i-1,1}\o |^{e-i})\circ\alt_e
 = (-1)^{(1,2,\dots,i-1)} \alt_e
 = (-1)^i \alt_e .
$$
\end{proof}

\begin{definition}
We describe an operator on arc diagrams as follows: for $1\le k\le d$, $r\in\{\pm1\}$ and each arc diagram $x$ with $2d$ lower points, let $C_{k,d}^r(x)$ be the arc diagram 
$$
C_{k,d}^r(x) := |^{2d-1} \cup |^{\ell-2d+1} \circ(2d-1,2d)^{\delta_{r,-1}} \circ(d,d-1,\dots,k)^{(2)}\circ (x\o |^{e}) .
$$
We denote the linear extension of this operator to linear combinations of such arc diagrams with the same symbol.
\end{definition}

In words, this is the arc diagram resulting from $x$ after adding $e$ vertical lines to the right, yielding $e$ new lower points, and then connecting a lower point in the $k$-th pair of adjacent lower points in $x$ with the first new lower point with a new arc, and moving the other point of the $k$-th pair to this position, where the sign $r$ determines the roles of the two lower points in the $k$-th pair.

\begin{definition} \label{def::I-omega}
For all $d\ge1$, we set
\begin{align*}
\omega(x) := \omega_d(x) 
&:= \sum_{k=1}^d (s_{d-1}\o\cyc)\circ( C_{k,d}^+(x) - C_{k,d}^-(x) )|_{S^{3\rho W}} 
\tforall x\in\Hom(S^{2\rho}W, S^d\gf)
.
\end{align*}
\end{definition}

As the index $d$ for $\omega$ can be inferred from the argument $x$, we will often omit it.

\begin{proposition} \label{prop::omega_d}
Write $\k=\sum_d \k_d=\sum_d s_d\circ\k'_d$ for $\k'\in\Hom(S^{2\rho}W,T^{2d} V)$. Then
$J(\k)=0$ if and only if $\omega(\k'_d)=0$ for all $d\ge1$.
\end{proposition}

\begin{proof} By \Cref{lem::Jk-1} and \Cref{lem::action-phi-e}, $J(\k)=0$ if and only if
$$
(|^{2d-2}\o\cyc)_{S^{d-1}\gf\o W}
\circ(|^{2d-1}\cup|^{e-1})_{S^d\gf\o W}
\circ((s_d\circ\k'_d)\o W)|_{S^{3\rho}W}=0 
\tforall d\ge 1.
$$
Set $G_d:=S_2\wr S_d=S_2^d\rtimes S_d$ and define the character $\chi_d\:G_d\to\{\pm1\}$ by $\chi_d((\tau,\theta))=(-1)^\tau$ for $\tau\in S_2^d$ and $\theta\in S_d$. Then
$$
s_d = \sum_{\s\in G_d} \chi_d(\s) \s .
$$
We have a natural embedding of $G_{d-1}$ into $G_d$, and using the corresponding coset decomposition, we can write any $\s\in G_d$ uniquely as $\s=\s' ((2d-1,2d)^r, (d,d-1,\dots,k))$ with $0\le r\le 1$ and $1\le k\le d$, where $(2d-1,2d)\in S_2^d\subset S_{2d}$ and $(d,d-1,\dots,k)\in S_d$ in the semidirect product $G_d$. In this case,
\begin{align*}
& (|^{2d-1}\cup|^{e-1})\circ (\chi_d(\s)\s\o W)
\\
&= (-1)^r \chi_{d-1}(\s') (\s'\o W) \circ (|^{2d-1}\cup|^{e-1})\circ (((2d-1,2d)^r,(d,d-1,\dots,k))\o W)  
\\
&= (-1)^r \chi_{d-1}(\s') (\s'\o W) \circ C^{(-1)^r}_{k,d}(|^{2d})
.
\end{align*}
So $J(\k)=0$ if and only if
$$
\sum_{k=1}^d (s_{d-1}\o\cyc) 
\circ (C_k^+(\k'_d) - C_k^-(\k'_d))|_{S^{3\rho} W}
=0 
\tforall d\ge 1,
$$
as asserted.
\end{proof}

\section{A Jacobi calculus based on pseudographs}

Again, we fix $e\ge1$, $W:=\L^e V$ and $\rho\in\{\pm1\}$. We have seen in \Cref{prop::omega_d} that, in order to determine interpolation PBW deformations, we should compute $\omega(x)$ for $x$ ranging over a suitable set of basis vectors in the relevant hom-spaces. 

Recall the definition of the distinguished morphisms $x(\g,\l)$ in $\RepOt$ (see \Cref{def::x-g-l}). We have seen in \Cref{prop::Hom-sym} that certain symmetrizations of them form a basis of the hom-spaces $\Hom(S^{\pm2}\L^e V,S^d\L^2V)$. To describe the images of these basis elements under $\omega$, we define three constructions for $e$-valent $\NN$-edge-labelled pseudographs: 

\begin{definition} \label{def::modifications}
For $\g\in\Gamma_{2,e}$, $1\le i\le e$, and $1\le k<|\g_i|$, and $\ell\ge0$, let $\g^{(i,k)}, \g^{(i,-k)}, \g^{(\ell)}\in\Gamma_{4,e}$ be the $\NN$-edge-labelled pseudographs with vertices $\{1,2,3,4\}$ and with $2e$ $\NN$-edge-labelled edges (defined as elements in $\{1,\dots,4\}^2\times\NN$) 
$$
\g^{(i,k)}_j:=\begin{cases}
\g_j & 1\le j\le e, j\neq i \\
(\g_{i,1},4,k-1) & j=i \\
(\g_{i,2},3,|\g_i|-k) & j=e+1 \\
(3,4,0) & e+1 \le j \le 2e
\end{cases}, 
\qquad
\g^{(i,-k)}_j:=\begin{cases}
\g_j & 1\le j\le e, j\neq i \\
(\g_{i,1},3,k-1) & j=i \\
(\g_{i,2},4,|\g_i|-k) & j=e+1 \\
(3,4,0) & e+1 \le j \le 2e
\end{cases},
$$
$$
\tand\quad
\g^{(\ell)}_j:=\begin{cases}
\g_j & 1\le j\le e \\
(3,4,\ell-1) & j=e+1 \\
(3,4,0) & e+1 \le j \le 2e
\end{cases}
\qquad,
$$
respectively.
\end{definition}

\begin{remark} Note that all constructions yield pseudographs in $\Gamma_{4,e}$, not just in $\ti\Gamma_{4,e}$, as no new loops are produced. 
\end{remark}

We give a few visualizations of these constructions.

\newcommand\mybox{\node[rectangle,draw=gray,anchor=center,inner xsep=0.65cm,inner ysep=0.25cm] at (1.5,1) {...};}
\begin{example} \label{expl::gamma-i-k}
For $\g$ as in the definition and $\g_i$ an edge between vertices $1$ and $2$,
$$
\g^{(i,k)}_j = \myxmoredots{3}{1}{ 
(a3) to[bend right=20] node[below right=-1pt]{...} (b1)
(a3) to[bend left=40] node[below right]{$0$ ($e-1$ times)} (b1)
(a2) to[bend left] node[above right]{$|\g_i|-k$} (a3)
(a1) to[bend right] node[below left]{$k-1$} (b1);
\mybox
},
\qquad
\g^{(i,-k)}_j = \myxmoredots{3}{1}{ 
(a3) to[bend right=20] node[below right=-1pt]{...} (b1)
(a3) to[bend left=40] node[below right]{$0$ ($e-1$ times)} (b1)
(a1) to[bend left=60] node[above right]{$k-1$} (a3)
(a2) to[bend right] node[below left]{$|\g_i|-k$} (b1);
\mybox
},
$$
where the boxed part $\myxdots{2}{;\mybox}$ shall represent the graph which results from $\g$ after removing the edge $\g_i$.
\end{example}

\begin{example} For $\g,\ell$ as in the definition, 
$$
\g^{(\ell)} = \myxmoredots{3}{1}{
(a3) to[bend right=10] node[above]{$\ell$} (b1)
(a3) to[bend left=20] node[below right=-1pt]{...} (b1)
(a3) to[bend left=110] node[below right]{$0$ ($e-1$ times} (b1);
\mybox
} ,
$$
where the boxed part $\myxdots{2}{;\mybox}$ shall now represent $\g$.
\end{example}

\begin{example} \label{expl::graph-mods-1}
For $e=1$, consider: $\g = \myxdots{2}{
(a1) -- node[above] {3} (a2)
}$,
$$
\g^{(1,3)} = \myxmoredots{3}{1}{ 
(a2) -- node[above] {0} (a3)
(a1) -- node[below left] {2} (b1)
}
\qquad
\g^{(1,-1)} = \myxmoredots{3}{1}{
(a1) to[bend left] node[above] {0} (a3)
(a2) -- node[right] {2} (b1)
}
\qquad
\g^{(5)} = \myxmoredots{3}{1}{
(a1) -- node[above] {3} (a2)
(a3) -- node[below right] {4} (b1)
}
$$
\end{example}

\begin{example} \label{expl::graph-mods-2}
For $e=3$ and with edges $\g_1,\g_2,\g_3$ listed from left to right in the picture of $\g$, consider:
$$
\g = \myxdots{2}{ (a1) -- node[above] {1} (a2); \myloop{a1}{0} \myloop{a2}{2} 
}
\qquad
\g^{(3,1)} = \myxmoredots{3}{1}{
(a1) -- node[above] {1} (a2)
(a2) -- node[left] {0} (b1)
(a2) -- node[above] {1} (a3)
(a3) -- node[left] {0} (b1)
(a3) to[bend left] node[below right] {0} (b1);
\myloop{a1}{0}
}
\qquad
\g^{(4)} = \myxmoredots{3}{1}{ 
(a1) -- node[above] {1} (a2)
(a3) to[bend right=45] node[below] {3} (b1)
(a3) -- node[below] {0} (b1)
(a3) to[bend left=45] node[below] {0} (b1);
\myloop{a1}{0} \myloop{a2}{2}
}
$$
\end{example}

\begin{remark} \label{rem::fourth-point}
Recall from \Cref{cor::Hom-tensor} that a basis of $\Hom(T^4\L^eV,\Ug)$ is parameterized by pairs $(\g,\l)$, where $\g\in\G_{4,e}$ and $\l\in P^+$. Dualizing one copy of the object $\L^e V$, we obtain a parametrization of $\Hom(T^3\L^eV,\Ug\o\L^eV)$ by the same datum; in the graphical representations of elements of $\G_{4,e}$ above we have already drawn the fourth vertex in a second row below the other vertices to emphasize it is representing a tensor factor of the target object, not of the source object.
\end{remark}

\begin{definition} \label{def::I-and-y}
Using the distinguished morphism $\cyc$ from \Cref{def::cyc}, we define
$$
I_d(y) := (s_{d-1}\o\cyc)\circ y|_{S^{3\rho}\L^e V}
\tforall y\in\Hom(T^3\L^eV, T^{2d-2}V\o\L^e V) ,
$$
and, as a short-hand, we set $y_d(\g',\l):=I_d(x(\g',\l))$, an element in $\Hom(S^{3\rho}\L^e V,\Ug\o\L^e V)$ for all $\g'\in\G_{4,e}$ and $\l\in P^+$.
\end{definition}

\begin{lemma} For all $\g\in\Gamma_{2,e}$, $\l\in P^+$, $\g_i$ an edge of $\g$, $\l_j$ a part of $\l$, set $x:=x(\g,\l)$, $d:=|\g|+|\l|$, $p_i:=|\g_1|+\dots+|\g_{i-1}|$, and $q_j:=|\g|+\l_1+\dots+\l_{j-1}$. Then 
\begin{align}
\label{eq::pf-1} 
I_d(C_{p_i+m}^\pm(x))
&= (-1)^{|\g_i|-m} y_d(\g^{(i,\pm m)},\l) ,
\tforall 1\le m\le |\g_i|,
\\
\label{eq::pf-2}
 \pm I_d(C_{q_j+m}^\pm(x))
 &= y_d(\g^{(\l_j)},\l-\l_j) 
\tforall 1\le m\le\l_j .
\end{align}
\end{lemma}

\begin{proof} 
We observe that on each side of each asserted identity, we have a morphism which is up to a sign the symmetrization of a single arc diagram. Hence, to verify the identities, it suffices to check that the arc diagrams involved in both sides of each equation have the same arc sequence datum (by \Cref{lem::ASD}), and that permutations on the upper and lower points relating them evaluated with $\chi=\chi_{2,e,d}$ produce the asserted signs.

By construction of $x=x(\g,\l)$, for all $\g_i,p_i,p$ as in the statement, the arcs sequence of $x$ passing through the lower points $2(p_i+p)-1$ and $2(p_i+p)$ is of the form
$$
|^{(\g_{i,1})}\cap^{|\gamma_i|-1}|^{(\g_{i,1})}
$$
where the indices $\g_{i,1},\g_{i,2}\in\{1,2\}$ in parentheses shall indicate that the first and the last arc are connected to the groups $B_{\g_{i,1}}$ and $B_{\g_{i,2}}$ of $e$ upper point, respectively. Those arcs form an arc sequence of the form $(\g_{i,1},\g_{i,2},|\gamma_i|)$ in $x$. In the modified arc diagram $C_{p_i+p}^+(x)$, the lower point $2(p_i+p)$ is connected with $B_3$, and the lower point $2(p_i+p)-1$ is connected with $B_4$. The resulting arcs can be represented as
$$
|^{(\g_{i,1})} \cap^{p-2} |^{(4)} |^{(3)} \cap^{|\gamma_i|-p-1} |^{(\g_{i,2})},
$$
where the arcs $|^{(\g_{i,1})}\cap^{p-2}|^{(4)}$ degenerate to a single arc between $B_{\g_{i,1}}$ and $B_4$ if $p=1$, and similarly, the arcs $|^{(3)}\cap^{|\g_i|-p-1}|^{(\g_{i,2})}$ degenerate to a single arc between $B_{\g_{i,2}}$ and $B_3$ if $p=|\g_i|$.

As a consequence, the arc sequence $(\g_{i,1},\g_{i,2},|\gamma_i|)$ is replaced by two arc sequences $(\g_{i,2},3,|\gamma_i|-p)$ and $(\g_{i,1},4,p-1)$.  Beyond this change in the arc sequence datum, $C_{p_i+p}^+(x)$ also has an additional set of $e-1$ arcs between $B_3$ and $B_4$ stemming from the evaluation morphism. All other arcs in $x$ are left unchanged. Hence, the arc sequence data on both sides agree.

However, the order of the lower points of the right part
$$
|^{(3)}\cap^{|\g_i|-p-1}|^{(\g_{i,2})}
$$
is inverted compared to the order of the corresponding lower points in $x(\g^{(i,p)},\l)$, and the permutation relating them inverts the order in $|\g_i|-p$ pairs of lower points, hence producing a sign $(-1)^{|\g_i|-p}$. This proves \Cref{eq::pf-1} for the sign $\pm=+$.

Similarly, the effect of $C_{p_i+p}^-$ on $x$ is described by the modification
$$
|^{(\g_{i,1})}\cap^{|\gamma_i|-1}|^{(\g_{i,1})}
\mapsto
|^{(\g_{i,1})} \cap^{p-2} |^{(3)} |^{(4)} \cap^{|\gamma_i|-p-1} |^{(\g_{i,2})} ,
$$
but now we permute $B_1$ and $B_2$ by pre-composing with $\tau^{(e)}$, proving \Cref{eq::pf-1} for the sign $\pm=-$.

Analogously, for $\l_j,q_j,p$ as in the above summation, the effect of $C_{q_j+p}^+$ and $C_{q_j+p}^-$ on $x$ can be represented as follows:
\begin{align*}
    \arch{\cap^{\l_j-1}} 
    &\mapsto
    \arch{\cap^{p-2}|^{(4)}|^{(3)}\cap^{\l_j-p-1}} ,
    \\
    \arch{\cap^{\l_j-1}} 
    &\mapsto
    \arch{\cap^{p-2}|^{(3)}|^{(4)}\cap^{\l_j-p-1}}.
\end{align*}
On the left-hand side, we see an arc sequence with datum $(0,0,\l_j)$, while on the right-hand side, there is an arc sequence with datum $(3,4,\l_j-1)$. The order of all $2|\l_i|$ lower points in the second right-hand side is inverted and $|\l_i|$ is odd, proving \Cref{eq::pf-2}.
\end{proof}

\begin{definition}
For a graph $\g'\in\G_{4,e}$, let $\ol{\g'}$ be the graph with the vertices $1,2$ interchanged.
\end{definition}

\begin{lemma} \label{lem::modifications}
Consider $\g\in\G_{2,e}$, $1\le i\le e$, $1\le k\le|\g_i|$, and $\l\in P^+$. Set $d:=|\g|+|\l|$.

(a)  If $\g_i$ is a loop, i.e., $\g_{i,1}=\g_{i,2}$, then 
$$
y_d(\g^{(i,-k)},\l)
 = y_d(\g^{(i,|\g_i|+1-k)},\l)
.
$$

(b)  If $\g_i$ is not a loop, i.e., $\g_{i,1}\neq\g_{i,2}$, then $\g^{(i,-k)}=\ol{\left((\g^t)^{i,|\g_i|+1-k}\right)}$ and
$$
y_d(\g^{(i,-k)},\l)
 = \rho \sgn(\g) (-1)^{|\g_i|} 
 y_d((\g^t)^{(i,|\g_i|+1-k)},\l)
.
$$
\end{lemma}

\begin{proof} (a) It follows directly from the definitions of $\g^{(i,-k)}$ and $\g^{(i,|\g_i|+1-k)}$ that the graphs coincide (up to the order of the listing of the edges).

(b) Similar to (a), we note that the graphs $\g^{(i,-k)}$ and $\ol{\left((\g^t)^{i,|\g_i|+1-k}\right)}$ agree by definition. The last identity is shown using an argument analogous to that in the proof of \Cref{lem::sign}: the two arc diagrams $x(\g^{(i,-k)},\l)$ and $x((\g^t)^{(i,|\g_i|+1-k)},\l)\circ\tau^{(e)}$ have the same arc sequence datum, and a permutation in $G_{4,e,d}$ relating them is given by inverting the order of lower points along all arc sequences between $B_1$ and $B_2$. Those arc sequences are the same as the corresponding arc sequences, except for the arc sequence corresponding to $\g_i$, which is replaced by sequences not between $B_1$ and $B_2$. Hence, the sign of a permutation relating the two arc diagrams is given by $\sgn(\g)(-1)^{|\g_i|}$. Permuting the upper points produces a factor $\rho$.
\end{proof}

Note that the argument in the proof of \Cref{lem::sign} can be visualized using some of the pictures in \Cref{expl::graph-mods-1} and \Cref{expl::graph-mods-2}.

\newcommand\eqOmega{
 \omega(x(\g,\l))
 &= 
 \sum_{\g_i:\g_{i,1}\neq\g_{i,2}}
 \sum_{m=1}^{|\g_i|} (-1)^{|\g_i|+m} y_d(\g^{(i,m)},\l) 
 + \rho \sgn(\g) (-1)^{|\g_i|+m} y_d((\g^t)^{(i,m)},\l) )
 \\
 &\qquad\qquad + 2 \sum_{\l_j} \l_j y_d(\g^{(\l_j)},\l-\l_j) 
}
 
\begin{proposition} \label{prop::omega}
For all $\g\in\Gamma_{2,e}$ and $\l\in P^+$, set $d:=|\g|+|\l|$, then
\begin{align} \label{eq::omega}
 \eqOmega
 . \notag
\end{align}
\end{proposition}

\begin{proof}
By \Cref{def::I-omega}, with $x=x(\g,\l)$,
\begin{equation} \label{eq::omega-sym-x-tmp}
\omega(x) = I_d \Big(
\sum_{\g_i}
\sum_{m=1}^{|\g_i|} C_{p_i+m}^+(x) - C_{p_i+m}^-(x) 
+\sum_{\l_j}
\sum_{m=1}^{\l_i} C_{q_j+m}^+(x) - C_{q_j+m}^-(x) 
\Big) ,
\end{equation}
where $p_i:=|\g_1|+\dots+|\g_{i-1}|$ and $q_j:=|\g|+\l_1+\dots+\l_{j-1}$.

Assume $\g_i$ is a loop in $\g$. Then $|\g_i|$ is odd and by \Cref{lem::modifications}, $\g^{(i,-m)}$ and $\g^{(i,|\g_i|+1-m)}$ have the same image under $I_d$ for all $1\le m\le|\g_i|$. Hence, using \Cref{eq::pf-1}, we compute
\begin{align*}
\sum_{m=1}^{|\g_i|} C^+_{p_i+m}(x) - C^-_{p_i+m}(x)
 &= (-1)^{|\g_i|} \sum_{m=1}^{|\g_i|} (-1)^m 
 (y_d(\g^{(i,m)},\l) - y_d(\g^{(i,-m)}),\l)
 \\
 &= - \sum_{m=1}^{|\g_i|} (-1)^m 
 ( y_d(\g^{(i,m)},\l) - y_d(\g^{(i,|\g_i|+1-m)},\l) )
 \\
 &= - \sum_{m=1}^{|\g_i|} (-1)^m 
 ( y_d(\g^{(i,m)},\l) - (-1)^{|\g_i|+1} y_d(\g^{(i,m)},\l) )
 = 0 .
\end{align*}
Assume $\g_i$ is not a loop, so $\g_{i,1}=1$ and $\g_{i,2}=2$. Then again using \Cref{eq::pf-1} and \Cref{lem::modifications},
\begin{align*}
\sum_{m=1}^{|\g_i|} C^+_{p_i+m}(x) - C^-_{p_i+m}(x)
 &= (-1)^{|\g_i|} \sum_{m=1}^{|\g_i|} 
 (-1)^m y_d(\g^{(i,m)})
 - (-1)^m y_d(\g^{(i,-m)})
 \\
 &= (-1)^{|\g_i|} \sum_{m=1}^{|\g_i|} 
 (-1)^m y_d(\g^{(i,m)})
 - \rho (-1)^{m+|\g_i|} \sgn(\g) y_d((\g^t)^{(i,|\g_i|+1-m)})
 \\
 &= (-1)^{|\g_i|} \sum_{m=1}^{|\g_i|} 
 (-1)^m y_d(\g^{(i,m)}) 
 + \rho \sgn(\g) (-1)^m y_d((\g^t)^{(i,m)}) .
\end{align*}
The summands of the sum $\sum_{\l_i}(\dots)$ in \Cref{eq::omega-sym-x-tmp} are given by \Cref{eq::pf-2}, which completes the proof.
\end{proof}

\begin{example} For $\g=\myxdots{2}{
(a1) -- node[above] {1} (a2); \myloop{a1}{0} \myloop{a2}{2} 
}\in \G_{2,3}$ with edges $\g_1,\g_2,\g_3$ from the left to the right and $\l=(2,2)\in P^+$, we compute $d=|\g|+|\l|=7$ and
\begin{align*} 
\omega(x(\g,\l)) &= y_7( \g^{(2,1)}, \l ) 
 + y_7( (\g^t)^{(2,-1)}, \l )
 + 4 y_7( \g^{(2)}, \l-2 )
\\
&= y_7\Big( \myxmoredots{3}{1}{
    (a1) to node[below left]{$0$} (b1)
    (a2) to node[above]{$0$} (a3)
    (a3) to node[right=1pt]{$0$} (b1)
    (a3) to[bend left=70] node[right=1pt]{$0$} (b1)
    ; \myloop{a1}{0} \myloop{a2}{2} 
   } , (2,2) \Big) 
 + y_7\Big(  \myxmoredots{3}{1}{
    (a1) to node[below left]{$0$} (b1)
    (a2) to node[above]{$0$} (a3)
    (a3) to node[right=1pt]{$0$} (b1)
    (a3) to[bend left=70] node[right=1pt]{$0$} (b1)
    ; \myloop{a1}{2} \myloop{a2}{0} 
   }, (2,2) \Big)
\\
 &+ 4 y_7\Big( \myxmoredots{3}{1}{
    (a1) -- node[above] {1} (a2)
    (a3) to[bend right=10] node[left=1pt]{$1$} (b1)
    (a3) to[bend left=10] node[right=1pt]{$0$} (b1)
    (a3) to[bend left=70] node[right=1pt]{$0$} (b1)
    ; \myloop{a1}{0} \myloop{a2}{2} 
    }, (2) \Big) .
\end{align*}
\end{example}

We record the following version of \Cref{prop::omega} for symmetric graphs $\g$:

\begin{corollary} \label{cor::omega-sym}
For all $\g\in\Gamma^\rho_{2,e}$ with $\g^t=\g$ and $\l\in P^+$, set $d:=|\g|+|\l|$, then
\begin{align} \label{eq::omega-sym}
 \omega(x(\g,\l))
 =  2 \sum_{\g_i:\g_{i,1}\neq\g_{i,2}}
 \sum_{m=1}^{|\g_i|} (-1)^{|\g_i|+m} y_d(\g^{(i,m)},\l) 
 + 2 \sum_{\l_j} \l_j y_d(\g^{(\l_j)},\l-\l_j) 
 .
\end{align}
\end{corollary}

\section{Classification of interpolating PBW deformations}

Again, we fix $\rho\in\{\pm1\}$ and $e\ge1$.

Recall that we have derived a formula (\Cref{eq::omega}) for the images of certain basis elements $x^\rho(\g,\l)$ under the map $\omega$ in \Cref{prop::omega}. We will now determine, which linear combinations of these images are zero, which yields a classification of the interpolating PBW deformations of type $(\Ug,\L^eV,\rho)$. 

We start by inspecting \Cref{eq::omega}, which we recall for the convenience of our readers:
\begin{align} \label{eq::recall}
\eqOmega 
\quad \tforall \g\in\G_{2,e},\l\in P^+,d:=|\g|+|\l|
, \notag   
\end{align}
where according to our \Cref{def::I-and-y},
$$
y_d(\g',\l) = (s_{d-1}\o\cyc)\circ x(\g',\l)|_{S^{3\rho}\L^e V}
\tforall \g'\in\G_{4,e},\l\in P^+
;
$$
note that $|\g'|+|\l'|=d-1$ for all combinations $(\g',\l')$ appearing in the right-hand side of \Cref{eq::recall}.

As the expression is a morphism in $\Hom(S^{3\rho}\L^3V,\Ug\o\L^eV)$ and as $S^{3\rho}\L^eV$ is a direct summand in $T^3\L^eV$, we may describe it using the basis elements found in \Cref{cor::Hom-tensor}, specifically, in terms of basis elements of the form $x(\g',\l)$, where $\g'\in\G_{4,e}$ and $\l\in P^+$ (see \Cref{rem::fourth-point}).

In the following lemmas, we will use the action of $S_3$ on $\G_{4,e}$ which permutes the first three vertices.

\begin{lemma} \label{lem::coefficients}
Consider $\g',\ti\g'\in\G_{4,e}$ and $\l,\ti\l\in P^+$, set $d:=|\ti g'|+|\ti\l|$. Then the coefficient of $x(\g',\l)$ in $y_d(\ti\g',\ti\l)$ is zero unless $|\g'|+|\l|=d$, $\l=\ti\l$ and there is a permutation $\s\in S_3$ such that $\s\cdot\g'=\ti\g'$, in which case the coefficient is $\rho^\s$.
\end{lemma}

\begin{proof} The basis elements appearing in $y_d(\ti\g',\ti\l)$ are given by the orbit of $x(\ti\g',\ti\l)$ under the action of the group $G_{4,e,d}$ from \Cref{sec::parametrizing}, which leaves the associated graph and partition invariant, and the group $S_3$ permuting the three copies of $\L^e V$, which corresponds to the permutation of the first three vertices in the graph $\ti\g'$. 
\end{proof}

\begin{lemma}
Consider $\g,\ti\g\in\G^\rho_{2,e}$ and $\l,\ti\l\in P^+$. 
Assume $\ti\g_{\ti i}$ is an edge of $\ti\g$ between the vertices $1$ and $2$ and (a) $e\ge3$ or (b) $e=2$ and $|\ti\g_{\ti i}|>1$.
Then the coefficient of $\ti x:=x(\ti\g^{(\ti i,|\ti\g_{\ti i}|)},\ti\l)$ in $\omega(x(\g,\l))$ is zero unless $\ti\g\in\{\g,\g^t\}$ and $\ti\l=\l$, in which case it is $m_\g(\g_{\ti i})$ up to a sign and a power of $2$.
\end{lemma}

\begin{proof} 
Set $\xi:=\s\cdot\ti\g^{(\ti i,|\ti\g_{\ti i}|)}$ for any $\s\in S_3$. 

Due to our assumption (a) or (b), $\ti\g^{(\ti i,|\ti g_{\ti i}|)}$ is a graph with edges between vertex number $4$ and two other vertices $1\le v_0,v_1\le 3$. Hence, this is also the case for $\xi$. However, this is not the case for any graph of the form $\g^{(\ell)}$. So by \Cref{lem::coefficients}, the coefficient of $\ti x$ in the first sum in \Cref{eq::recall} is zero.

Furthermore, due to our assumption (a) or (b), one of the vertices $v_0$, $v_1$ is distinguished in that it has more than one edge or an edge with non-zero label connecting it to the vertex $4$. Let us assume this is the vertex $v_0$, and let $v_2$ be the remaining vertex which is not in $\{4,v_0,v_1\}$. Let $M$ be the sum of the labels of all edges in $\xi$ which are not between $v_1$ and $v_2$. Now $\xi$ can be of the form $\g^{(i,m)}$ only if $\g$ or $\g^t$ is the restriction of the graph $\xi$ to the vertices $v_1,v_2$, plus one edge between its vertices of label $M$. This graph is exactly $\ti\g$. Similarly, $\xi$ can be of the form $(\g^t)^{(i,m)}$ only if $\ti\g\in\{\g,\g^t\}$.

So the coefficient in question is zero unless $\ti\g=\{\g,\g^t\}$, but then \Cref{eq::recall} and \Cref{lem::coefficients} enforce $\l=\ti\l$, as well, and \Cref{eq::recall} then immediately implies that the coefficient has the desired form. 
\end{proof}

\begin{lemma}
Consider $\g,\ti\g\in\G^\rho_{2,e}$ and $\l,\ti\l\in P^+$.
Assume $\ti\l_{\ti j}$ is a (non-zero) part of $\ti\l$ and $e\ge2$.
Then the coefficient of $x(\ti\g^{(\ti\l_{\ti j})},\ti\l-\ti\l_{\ti j})$ in $\omega(x(\g,\l))$ is zero unless $\ti\g\in\{\g,\g^t\}$ and $\ti\l=\l$, in which case it is $m_\l(\l_{\ti j})$ up to a sign and a power of $2$.
\end{lemma}

\begin{proof} For any $\s\in S_3$, $\xi:=\s\cdot\ti\g^{(\ti\l_{\ti j})}$ is a graph with edges between vertex $4$ and exactly one other vertex $v_0$, which is not the case for any graph of the form $\g^{(i,k)}$ or $(\g^t)^{(i,k)}$. Now $\xi$ can be of the form $\g^{(\l_j)}$ only if $\g$ is the graph with two vertices with the same labelled edges as the ones among $v_1, v_2$ in $\xi$, so $\ti\g\in\{\g,\g^t\}$, and if $\l_j$ is the unique non-zero edge label between the vertices $4$ and $v_0$ in $\xi$, which is $\ti\l_{\ti j}$. So by \Cref{lem::coefficients}, the coefficient in question is zero unless $\g=\ti\g$ and $\l=\ti\l$, but then \Cref{eq::recall} immediately implies that the coefficient has the desired form. 
\end{proof}

Recall from \Cref{def::x-rho} that
$$
x^\rho(\g,\l) = x(\g,\l) + \rho \sgn(\g) x(\g^t,\l) .
$$

\begin{corollary} \label{cor::data-solution}
$J(\k)=0$ implies that $\k$ is a linear combination of $x^\rho(\g,\l)$ for $\g\in\G^\rho_{2,e}, \l\in P^+$ such that one of the following conditions is met:

(a) $e=1$

(b) $e=2$, all edges in $\g$ between $1$ and $2$ have label at most $1$, and $\l=\emptyset$

(c) $e\ge3$, $\g$ has no edges between $1$ and $2$, and $\l=\emptyset$
\end{corollary}

\begin{proof} The lemmas show that the image under $\omega$ of linear combinations involving  $x^\rho(\g,\l)$ beyond what is described in (a) -- (c) are non-zero.
\end{proof}

\begin{definition} Let $\J^\rho_e$ the space of all interpolating PBW deformations of type $(\Ug,\L^eV,\rho)$.
\end{definition}

For convenience, we will also write $\J^\pm$ for $\J^{\pm1}$.

We will see that \Cref{cor::data-solution} yields bases for the spaces $\J^\pm_e$ for $e>1$ immediately.

\begin{definition} \label{def::form-field} \label{def::gamma-mu-nu}
We set 
$$
\g^\form := \myxdots{2}{(a1) to[bend left] node[above]{$0$} (a2) (a1) to[bend right] node[below]{$0$} (a2)} \in\G^+_{2,2}, \quad
\g^\Lie := \myxdots{2}{(a1) to[bend left] node[above]{$1$} (a2) (a1) to[bend right] node[below]{$0$} (a2)}  \in\G^-_{2,0},
$$
$$
\k^\form:=x^+(\g^\form,\emptyset),\quad
\k^\Lie:=x^-(\g^\Lie,\emptyset) ,
$$
$$
\quad Q^\rho_{\ell} := \{(\mu,\nu)\in P^2: \#\mu=\#\nu=\ell  , \mu<\nu \text{ or } (\mu=\nu \tand \rho=1) \} 
\quad\tforall \ell\ge1, \rho\in\{\pm1\},
$$
and for all $(\mu,\nu)\in Q^\rho_\ell$, let $\g^{(\mu,\nu)}\in\G_{2,2\ell}$ be the graph with loops with labels $\mu_1,\dots,\mu_{\#\mu}$ and $\nu_1,\dots,\nu_{\#\nu}$ at its two vertices.
\end{definition}

\begin{proposition} 
$\J^+_2$ has a basis consisting of $\k^\form$ and $x^+(\g^{(\mu,\nu)},\emptyset)$ for $(\mu,\nu)\in Q_1^+$.

$\J^-_2$ has a basis consisting of $\k^\Lie$ and $x^-(\g^{(\mu,\nu)},\emptyset)$ for $(\mu,\nu)\in Q_1^-$.

For all $e\ge 3$, $\J^+_e=0$ if $e$ is odd, and $\J^\pm_e$ has a basis consisting of $x^\pm(\g^{(\mu,\nu)},\emptyset)$ for $(\mu,\nu)\in Q^\pm_{e/2}$ if $e$ is even.
\end{proposition}

\begin{proof} By \Cref{cor::data-solution}, the spaces of PBW deformations are spanned by linear combinations of the described deformation maps, where we have defined the sets $Q^\rho_\ell$ such that their symmetrizations via $x^+$ or $x^-$ are distinct and non-zero. 
So it suffices to see that the asserted basis elements all solve \Cref{eq::recall}. This is true for $x^{\pm}(\g^{(\mu,\nu)},\emptyset)$ and for $\k^\form$, as the sums on the right-hand side of \Cref{eq::recall} are empty. Finally, for $\k^\Lie$, the right-hand side of \Cref{eq::recall} is a multiple of 
$$
y_0(\myxmoredots{3}{1}{ 
(a1) -- node[above] {0} (a2)
(a2) -- node[above] {0} (a3)
(a1) -- node[below left] {0} (b1)
(a3) -- node[below right] {0} (b1)
},\emptyset) .
$$
However, this expression is zero, as reversing the order of the three upper vertices is a symmetry of the graph, and the expression is antisymmetric with respect to any transposition of upper vertices. 
\end{proof}

\bigskip

To complete the classification, we analyze the case $e=1$. For all $M\ge0, \l\in P^+, k\ge0$, we set
$$
x^\pm_{M,\l} := x^\pm(~\myxdots{2}{(a1) -- node[above]{$M$} (a2)}~, \l)
\quad\tand\quad
y_{M,\l,k}
:= y_{M+|\l|+k}\left(\myxmoredots{3}{1}{(a1) -- node[above]{$M$} (a2) (a3) -- node[below right]{$k-1$} (b1)}, \l\right)
.
$$

\begin{lemma}
(a) The set $\{s_{M+|\l|}\circ x^\rho_{M,\l}: M\ge0, \l\in P^+,(-1)^M=\rho\}$ is a basis of $\Hom(S^{2\rho} V,\Ug)$.

(b) $y_{M,\l,k}=0$ if $(-1)^M\neq\rho$.

(c) The set $\{y_{M,\l,k}: M\ge0,\l\in P^+,k\ge0,(-1)^M=\rho\}$ is linearly independent.
\end{lemma}

\begin{proof}
(a) is \Cref{prop::Hom-sym} for $e=1$.

To show (b), we set
$$
\g'_{M,k}:=\myxmoredots{3}{1}{(a1) -- node[above]{$M$} (a2) (a3) -- node[below right]{$k-1$} (b1)} .
$$
The arc diagram $x(\g'_{M,k},\l)$ has an arc sequence with two upper points and $2M$ lower points, so it is invariant under permuting the two upper points and reversing the order of all lower points. Hence, viewed as a morphism between the corresponding combination of symmetric and exterior powers, it is identical to itself times a factor of $\rho(-1)^M$, so it is zero, if $\rho\neq(-1)^M$.

Note that by the same argument, the coefficient of $x(\g'_{M,k}, \l)$ in the symmetrization $y_{M,\l,k}$ is non-zero if $(-1)^M=\rho$, so as the elements $\{x(\g'_{M,k},\l):M\ge 0,\l\in P^+,(-1)^M=\rho\}$ are linearly independent, so is the asserted set.
\end{proof}

\begin{lemma}
For all $M\ge0$, $\l\in P^+$ with $(-1)^M=\rho$,
\begin{align} \label{eq::omega-sym-e-1}
\omega^\rho(x^\rho_{M,\l})=
2\rho \sum_{1\le m\le M \text{ even}} 
y_{M-m,\l,m}
+2\sum_{\l_j} \l_j y_{M,\l-\l_j,\l_j}
.
\end{align}
\end{lemma} 

\begin{proof}
The formula is obtained by specializing $e=1$ in \Cref{eq::omega-sym}.
First, note that if $\g=\myxdots{2}{(a1) -- node[above]{$M$} (a2)}$, then for all $1\le m\le M$, $\ell\ge1$,
$$
\g^{(1,m)} = \myxmoredots{3}{1}{(a2) --node[above=4pt]{$M-m$} (a3) (a1) --node[below left]{$m-1$} (b1)} ,
\quad
\g^{(\ell)} = \myxmoredots{3}{1}{(a1) --node[above]{$M$} (a2) (a3) --node[below right]{$\ell-1$} (b1)} ,
$$
where the first graph is the one appearing in $y_{M-m,\l,m}$ up to a rotation of the three upper points, and the second graph is the one appearing in $y_{M,\l,\ell}$ or $y_{M,\l-\ell,\ell}$.
Hence, \Cref{eq::omega-sym} becomes
\begin{align*}
 \omega^\rho(x^\rho_{M,\l})
 &= 2 \sum_{m=1}^M (-1)^{M+m} y_{M-m,\l,m} 
 + 2 \sum_{\l_j} \l_j y_{M,\l-\l_j,\l_j} .
\end{align*}
Now $(-1)^M=\rho$ and $y_{M-m,\l,m}=0$ if $(-1)^{M-m}\neq\rho=(-1)^M$.
\end{proof}

For any partition $\l$ and $\ell>0$, let $m_\l(\ell)$ be the number of times $\ell$ appears as a part in $\l$.

\begin{definition} \label{def::kappa}
For each $p\ge0$, $(-1)^p=\rho$, set
$$
\k^{(\rho, p)} := \sum_{\nu\in P^+, |\nu|\le p}
\frac{(-\rho)^{\#\nu}}{\nu_1\dots\nu_{\#\nu} m_\nu(\nu_1)!\dots m_\nu(\nu_{\#\nu})!}
x^\rho_{w-|\nu|,\nu}
$$
\end{definition}

\begin{proposition} $\J^\rho_1$ is spanned by $\{\k^{(\rho,w)}: w\ge0, (-1)^w=\rho\}$.
\end{proposition}

\begin{proof} 
Write $\k:=\sum_{M\ge0,\l\in P^+} \alpha_{M,\l} x^\rho_{M,\l}$ with coefficients $\alpha_{M,\l}\in\kk$ only finitely many of which are non-zero. Then $J(\k)=0$ if and only if
\begin{equation} \label{eq::coefficients-e-1}
2 m_\l(\l_j) \l_j \alpha_{M,\l} = -2\rho \alpha_{M+\l_j,\l-\l_j}
\end{equation}
for all $M$, $\l$, and $j$, as \Cref{eq::omega-sym-e-1}, $x^\rho_{M,\l-\l_j,\l_j}$ appears with coefficient $2 m_\l(\l_j) \l_j$ in $\omega^\rho(x^\rho_{M,\l})$, with coefficient $2\rho$ in $\omega^\rho(x^\rho_{M+\l_j,\l-\l_j})$, and with coefficient $0$ otherwise.

Using this repeatedly, we see that $J(\k)=0$ only if
$$
\alpha_{w-|\nu|,\nu}
= \frac{-\rho}{\nu_1 m_\nu(\nu_1)} \alpha_{w-|\nu|+\nu_1, \nu-\nu_1}
=\dots
= \frac{(-\rho)^{\#\nu}}{\nu_1\dots\nu_{\#\nu} m_\nu(\nu_1)!\dots m_\nu(\nu_{\#\nu})!} \alpha_{s, \emptyset} .
$$
On the other hand, it follows immediately that $\k^{(\rho,w)}$ satisfies \Cref{eq::coefficients-e-1} for all $w\ge0$.
\end{proof}

Let us summarizing the results obtained in this section:

\begin{theorem} \label{thm::PBW-defs}
The space of interpolating PBW deformations $\J^\rho_e$ is the vector space spanned by
\begin{itemize}
\item $\{\k^{(\rho,w)}: w\ge0, (-1)^w=\rho\}$ if $e=1$,
\item $\{\k^\form\}\cup\{x^+(\g^{(\mu,\nu)},\emptyset):(\mu,\nu)\in Q_1^+\}$ if $e=2$, $\rho=1$,
\item $\{\k^\Lie\}\cup\{x^-(\g^{(\mu,\nu)},\emptyset):(\mu,\nu)\in Q_1^-\}$ if $e=2$, $\rho=-1$,
\item $\{x^\rho(\g^{(\mu,\nu)},\emptyset):(\mu,\nu)\in Q_{e/2}^\rho\}$ if $e>2$ even,
\end{itemize}
and the space of PBW deformations is zero in all other cases. \end{theorem}

\begin{proof}
We have shown that the asserted deformation maps span the respective spaces of PBW deformations. The linear independence of the sets follows from their definition using disjoint sets of basis vectors in the corresponding hom-spaces. 
\end{proof}

\section{Specialized PBW deformations} \label{sec::specialization}

Recall from \Cref{sec::RepOt} that we have a symmetric monoidal specialization functor
$$
F=F_{m,2n}\colon\RepOt\to\SVec_\kk
$$
for all $m,n\ge0$ for $t=m-2n$. If $n=0$, the image of $F$ is in fact in $\Vec_\kk$, in any case, it is a full subcategory of $\Rep G_{m,2n}$ for $G_{m,2n}:=\OSp(m|2n)$. By abuse of notation, we denote $F(V)\in\Rep G_{m,2n}$ by the symbol $V$ and $F(\gf)=F(\L^2V)\in \Rep G_{m,2n}$ (see \Cref{lem::interpolating-lie-algebra}) by the symbol $\gf$, as well, as long as the usage is unambiguous. Note that $\gf\in\Rep G_{m,2n}$ is the Lie (super)algebra of $G_{m,2n}$.

In the following, we compute the specializations of the PBW deformations found in \Cref{thm::PBW-defs} under the functor $F_{m,2n}$.

Two describe the functor $F_{m,2n}$ on the level of morphisms, we define distinguished morphisms in $\Rep G_{m,2n}$. Let $e_1,\dots,e_{m+2n}$ the standard basis of $V_\kk=\kk^{m|2n}$ such that $e_i$ is even for $1\le i\le m$ and odd for $m<i\le m+2n$. We set 
$$
e_i^*:=\begin{cases}
e_i & 1\le i\le m \\
e_{i+1} & m<i\le m+2n \todd \\
-e_{i-1} & m<i\le m+2n \teven
\end{cases} .
$$
The vectors $(e_i^*)_i$ form a basis of $V_\kk$, as well, since they are the same vectors as $(e_i)_i$ up to signs. The two bases are dual to each other with respect to a non-degenerate orthosymplectic form on $V_\kk$.

\begin{definition} \label{def::images-F}
We define
\begin{align*}
f_\cup &\: V_\kk\o V_\kk\to\kk, \quad 
e_i^*\o e_j \mapsto \delta_{ij} , \\
f_\cap &\: \kk\to V_\kk\o V_\kk, \quad
1\mapsto \sum_i e_i\o e_i^* , \\
f_\cross &\: V_\kk\o V_\kk\to V_\kk\to V_\kk, \quad
e_i\o e_j\mapsto (-1)^{|e_i| |e_j|} e_j\o e_i ,
\end{align*}
where we use the degrees $|e_i|,|e_j|\in\{0,1\}$ of homogeneous basis elements in the supervector space $V_\kk$.
\end{definition}

Note that by our definitions,
$$
f_\cap(1) = \sum_i e_i\o e_i^* = \sum_i (-1)^{|e_i|} e_i^*\o e_i .
$$

\begin{lemma} \label{lem::F-computations1}
$(V_\kk,f_\cup,f_\cap)$ defines an object with a symmetric self-duality of dimension $t$ in $\Rep G_{m,2n}$. Hence, we have specialization functors $F=F_{m,2n}$ sending $\cap\mapsto f_\cap$, $\cup\mapsto f_\cup$, and $\cross\mapsto f_\cross$.
\end{lemma}

\begin{proof} $f_\cross$ is the symmetric braiding in the category of supervector spaces. Hence, by the universal property of $\RepOt$ \cite{Del}*{Prop.~9.4} (see \Cref{sec::RepOt}), it suffices to check that $V_\kk$ is an object with a symmetric self-duality defined by $f_\cup,f_\cap$ and of dimension $t=m-2n$. This amounts to checking that
$$
(V_\kk\o f_\cup)(f_\cap\o V_\kk) = V_\kk = (f_\cap\o V_\kk)(V_\kk\o f_\cap) ,\quad
f_\cup = f_\cup f_\cross ,\quad
f_\cap = f_\cross f_\cap ,\quad
f_\cup f_\cap = t ,
$$
which is a straight-forward computation.
\end{proof}

To describe the images of certain morphisms under $F$, we define elements
$$
E_{ij} := e_i\o e_j^*,\quad
E^*_{ij} := e_i^*\o e_j \quad
\tforall 1\le i,j\le m+2n\quad \in V_\kk\o V_\kk ,
$$
so $(E_{ij})_{i,j}$ and $(E^*_{ij})_{i,j}$ are bases of $V_\kk\o V_\kk$.

\begin{lemma} \label{lem::F-computations2}
The following holds for all $m\ge1$:
\begin{align}
F(|\cup|) &= (E_{ij}\o E_{k\ell} 
\mapsto \delta_{jk} E_{i\ell} ) \label{eq::F-1} \\
F(\uarch{\cup}) &= (E_{ij}\o E_{k\ell} 
\mapsto (-1)^{|e_i|} \delta_{jk} \delta_{i\ell})  \label{eq::F-2} \\
F(| \cap^{m-1} |) &= (E^*_{ij}
\mapsto \sum_{i_2,\dots,i_m} 
E^*_{ii_2}\o E^*_{i_2i_3}\o\dots\o E^*_{i_mj} ) \label{eq::F-3} \\
F(\arch{\cap^{m-1}}) &= (1\mapsto 
\sum_{i_1,\dots,i_m} 
(-1)^{|e_{i_1}|} E^*_{i_1 i_2}\o E^*_{i_2 i_3}\o\dots\o E^*_{i_m i_1})  \label{eq::F-4}
\end{align}
\end{lemma}

\begin{proof} All these follow directly from \Cref{lem::F-computations1}.
\end{proof}

Using the form $f_\cup$ on $V_\kk$ to identify $V_\kk\simeq V_\kk^*$, we can view the elements $E_{ij}$ and $E^*_{ij}$ as matrices in $V_\kk\o V_\kk\simeq V_\kk\o V_\kk^*=\gl(V_\kk)=:\gl$ and compute matrix products and \emph{supertraces}, the latter ones being defined as
$$
\str(E_{ij}) := \delta_{ij} (-1)^{|e_i|} .
$$
The \emph{supertrace form} on $\gl$ is defined as the supersymmetric bilinear form $X\o Y\mapsto \str(XY)$, and the Lie bracket on $\gl$ is given by the supercommutator $X\o Y\mapsto XY-(-1)^{|X||Y|}YX$.
We can view $V\wedge V\simeq\gf$ as a subspace of $\gl$ and verify that it is, in fact, a Lie subalgebra.

\begin{lemma} Up to non-zero scalars, $F(\k^\form)\:S^2\gf\to\kk$ is the supersymmetric supertrace form and $F(\k^\Lie)\:\L^2\gf\to\gf$ is the Lie bracket.
\end{lemma}

\begin{proof} We observe that by definition (see \Cref{def::form-field}), 
$$
x(\g^\form,\emptyset) = \uarch{\cup} 
\qquad\tand\qquad
x(\g^\Lie,\emptyset) = |\cup| , 
$$
where, as always, arc diagrams are unique only up to an action of the corresponding symmetry groups. Now the assertions follow from \Cref{lem::F-computations2}.
\end{proof}

For $1\le i,j\le m+2n$, let us define elements 
\begin{align*}
X_{ij}
&:= E_{ij} - (-1)^{|e_i||e_j|} E^*_{ji} = e_i\o e^*_j-(-1)^{|e_i||e_j|}e^*_j\o e_i,
\\
X^*_{ij}
&:= E^*_{ij} - (-1)^{|e_i||e_j|} E_{ji}
=e^*_i\o e_j - (-1)^{|e_i||e_j|} e_j\o e^*_i
\quad \in\L^2 V\simeq\gf .
\end{align*}
By definition, $X_{ij}=-(-1)^{|e_i||e_j|} X_{ji}^*$.

Next, we note that for all even $e\ge2$, we may view $\L^e V$ as a direct summand in $T^e V \cong T^{e/2} T^2 V$.

\begin{definition} For all even $e\ge2$ and $\mu\in P$ with $\#\mu=e/2$, we define $\ti h_{\mu} \: T^{e/2} T^2V \to S^{|\mu|} \gf$ by
\begin{equation} \label{eq::ti-h}
\ti h_{\mu}\Big( \bigotimes_{k=1}^{e/2} E^*_{i_k,j_k} \Big)
:= \prod_{k=1}^{e/2} \Big( 
\sum_{1\le i'_2,\dots,i'_{\mu_k}\le m+2n} 
X^*_{i_k,i'_2} X^*_{i'_2,i'_3} \dots X^*_{i'_{\mu_k-1},i'_{\mu_k}} X^*_{i'_{\mu_k},j_k}
\Big) ,
\end{equation}
and we define $h_{\mu}$ to be the restriction of $\ti h_{\mu}$ to $\L^e V$.
\end{definition}

\begin{lemma} For all even $e\ge 2$ and $(\mu,\nu)\in Q^\rho_{e/2}$, $F(x^\rho(\g^{(\mu,\nu)}),\emptyset)$ is up to a non-zero scalar the morphism 
$$
h_{\mu} \o h_{\nu}+\rho h_{\nu} \o h_{\mu}
\colon S^{2\rho} \L^e V\to S^{|\mu|+|\nu|} \gf
.
$$
\end{lemma}

\begin{proof} This is \Cref{lem::F-computations2} applied to $x^\rho(\g^{(\mu,\nu)},\emptyset)$, where $\g^{(\mu,\nu)}$ is defined in \Cref{def::gamma-mu-nu}: any part of $\mu$ or $\nu$ corresponds to a labelled loop in the graph $\g^{(\mu,\nu)}$ which corresponds to one factor in the product on the right-hand side of \Cref{eq::ti-h} by \Cref{eq::F-4}.
\end{proof}

We also note that we may view $S^{2\rho} V$ as a direct summand in $T^2 V$.

\begin{definition} For all $k\ge0$, $(-1)^k=\rho$, and all $\nu\in P^+$, we define $\ti g_{k,\nu}\colon T^2V \to S^{k+|\nu|}\gf$ by
\begin{equation} \label{eq::g}
\ti g_{k,\nu}(E^*_{ij}):= \Big(
\sum_{1\le i_2,\dots,i_k\le m+2n} X^*_{i,i_2} X^*_{i_2,i_3}\dots X^*_{i_k,j}
\Big) \prod_{k=1}^{\#\nu} \Big( 
\sum_{1\le i_1,\dots,i_{\nu_k}\le m+2n} (-1)^{|e_{i_1}|} X^*_{i_1,i_2} X^*_{i_2,i_3}\dots X^*_{i_{\nu_k}, i_1} 
\Big).
\end{equation}
and we define $g_{k,\nu}$ to be the restriction of $\ti g_{k,\nu}$ to $S^{2\rho}V$.
\end{definition}

\begin{lemma} \label{lem::formula-kappa-1} For all $w\ge0$, $(-1)^w=\rho$, $F(\k^{(\rho,w)})$ is up to a non-zero scalar the morphism
\begin{equation} \label{eq::formula-kappa-1}
\sum_{\nu\in P^+, |\nu|\le w}
\frac{(-\rho)^{\#\nu}}{\nu_1\dots\nu_{\#\nu} m_\nu(\nu_1)!\dots m_\nu(\nu_{\#\nu})!}
g_{w-|\nu|,\nu} \: S^{2\rho} V\to S^w\gf\subset\Ug .
\end{equation}
\end{lemma}

\begin{proof} This is \Cref{lem::F-computations2} applied to $\k^{(\rho,w)}$ as defined in \Cref{def::kappa}: the edge with the label $w-|\nu|$ between two distinct vertices in each graph involved in the definition of $\k^{(\rho,w)}$ corresponds to the first sum on the right-hand side of \Cref{eq::g}, while the parts of $\nu$ correspond to arc sequences on lower points only in the arc diagrams in $x^\rho_{w-|\nu|,\nu}$, which correspond to the factors in the product on the right-hand side of \Cref{eq::g} by \Cref{eq::F-3}.
\end{proof}

We want to compare this with known results on PBW deformations over $\kk=\CC$. Recall that Tsymbaliuk classified PBW deformations of type $(O_m, \CC^m, -1)$ in \cite{Tsy}*{Thm.~1.4}.

\begin{proposition} \label{prop::cmp-Tsy} 
For $\kk=\CC$, the specialized PBW deformations $\{F_{m,0}(\k^{(-1,w)}):w\ge0\todd\}$ for all $m\ge0$ are the same as the ones found in \cite{Tsy}.
\end{proposition}

\begin{proof} The PBW deformations in \cite{Tsy} are defined as follows: first, we identify $S^w\gf\simeq S^w\gf^*$ using the trace form $X\o Y\mapsto\tr(XY)$. Then a basis of the PBW deformation maps is given by the maps sending $(v_1\o v_2)$ to the coefficient of $\t^{2N}$, for some $N\ge0$, in the expansion of
$$
A\mapsto (v_2, A(1+\t^2A^2)^{-1}v_1) \det(1+\t^2A^2)^{-1/2}
= (v_2, \sum_{k\ge0} (-1)^k \t^{2k} A^{2k+1} v_1)
 \exp(\sum_{k\ge1} \frac{(-1)^k}{2k} \t^{2k} \tr(A^{2k})) ,
$$
where the identity of power series $\det=\exp\circ\log\circ\det=\exp\circ\tr\circ\log$ is used. This coefficient can be expanded as
\begin{multline*}
\sum_{k=0}^N (-1)^k (v_2, A^{2k+1} v_1) \Big[
\sum_{\ell\ge0} \frac1{\ell!} \Big( 
\sum_{r\ge1} \frac{(-1)^r \t^{2r} \tr(A^{2r})}{2r} 
\Big)^\ell \Big]_{\tau^{2N-2k}}
\\
= \sum_{k=0}^N (-1)^k (v_2, A^{2k+1} v_1)
\sum_{\ell\ge 0} \frac1{\ell!} \Big(
\sum_{2r_1+\dots+2r_\ell=2N-2k}
\prod_{t=1}^{\ell} \frac{(-1)^{r_t} \tr(A^{2r_t})}{2r_t} 
\Big)
\\
= (-1)^N \sum_{k=0}^N  (v_2, A^{2k+1} v_1)
\sum_{\ell\ge0} \frac1{\ell!} \Big(
\sum_{\nu\in P^+, \#\nu=\ell, |\nu|=2N-2k} c_\nu
\prod_{t=1}^{\#\nu} \frac{\tr(A^{\nu_t})}{\nu_t} 
\Big)
\\
= (-1)^N \sum_{\nu\in P^+,|\nu|\le 2N}  (v_2, A^{2N+1-|\nu|} v_1)
\frac{c_\nu}{(\#\nu)!}
\prod_{t=1}^{\#\nu} \frac{\tr(A^{\nu_t})}{\nu_t} ,
\end{multline*}
where $c_\nu$ shall denote the number of distinct tuples whose entries are the parts of the partition $\nu$. This number is $c_\nu=\frac{(\#\nu)!}{m_\mu(\nu_1)!\dots m_\nu(\#\nu)!}$. The assertion follows if we identify $w$ with $2N+1$.
\end{proof}

Similarly, recall that Etingof--Gan--Ginzburg classified PBW deformations of type $(\Sp_{2n},\CC^{2n},-1)$ in \cite{EGG}*{Thm.~4.2}. We identify $\Sp_{2n}$ with the supergroup $\OSp(0|2n)$ and $\CC^{2n}$ with $F_{0,2n}(V)$. Note, however, that $F_{0,2n}(V)$ is a purely odd supervector space, so in the category of supervector spaces, we are considering PBW deformations for $\rho=+1$. 

\begin{proposition} \label{prop::cmp-EGG}
For $\kk=\CC$, the specialized PBW deformations $\{F_{0,2n}(\k^{(1,w)}):w\ge0\teven\}$ for all $n\ge0$ are the same as the ones found in \cite{EGG}.
\end{proposition}

\begin{proof} The proof is analogous to the one of \Cref{prop::cmp-Tsy}: the PBW deformations in \cite{EGG} are defined using the expansion of
$$
A\mapsto
(v_2, (1-\t^2A^2)^{-1} v_1)\det(1-\t A)^{-1}
= (v_2, \sum_{k\ge0} \t^{2k} A^{2k} v_1)
 \exp(-\sum_{k\ge1} \frac{1}{k} \t^k \tr(A^k)) .
$$
Taking the coefficient of $\t^{2N}$ for $N\ge0$ yields
\begin{multline*}
\sum_{k=0}^N (v_2, A^{2k} v_1) \sum_{\ell\ge0} \frac{1}{\ell!} 
\Big[ \Big(
-\sum_{r\ge1} \frac{\t^r \tr(A^r)}{r}
\Big)^\ell \Big]_{\t^{2N-2k}}
\\=
\sum_{k=0}^N (v_2, A^{2k} v_1) \sum_{\ell\ge0} \frac{(-1)^\ell}{\ell!} 
\sum_{r_1+\dots+r_\ell=2N-2k} \prod_{t=1}^\ell \frac{\tr(A^{r_t})}{r_t}
\\=
\sum_{\nu\in P,|\nu|\le 2N} (v_2, A^{2N-|\nu|} v_1)  \frac{(-1)^{\#\nu} c_\nu}{(\#\nu)!} \prod_{t=1}^{\#\nu} \frac{\tr(A^{\nu_t})}{\nu_t},
\end{multline*}
with $c_\nu=\frac{(\#\nu)!}{m_\mu(\nu_1)!\dots m_\nu(\#\nu)!}$ as above. The assertion follows after identifying $w=2N$ and noting that $\tr(A^r)=0$ for odd $r$.
\end{proof}

So, in particular, \Cref{eq::formula-kappa-1} gives a combinatorial formula for the PBW deformations found both in \cite{Tsy} and \cite{EGG}, which extends to all orthosymplectic supergroups.

\bibliography{bib}
\bibliographystyle{amsrefs}%

\end{document}